\documentclass[12pt,a4paper]{article}
\usepackage{latexsym,amssymb,amsmath,mathrsfs,amsthm}
\usepackage{sectsty}
\usepackage{upgreek}
\usepackage{color}
\usepackage{marvosym}
\usepackage{wasysym}
\usepackage{bm}
\usepackage[T1]{fontenc}
\usepackage[anchorcolor=blue,%
 bookmarks=true,%
 bookmarksnumbered=true,%
]{hyperref}
\usepackage{pst-all}
\usepackage{graphicx}

\allsectionsfont{\normalsize\sc\center}

\newtheorem{thm}{Theorem}[section]
\newtheorem{prop}[thm]{Proposition}
\newtheorem{cor}[thm]{Corollary}
\newtheorem{lem}[thm]{Lemma}

\newtheorem{defn}[thm]{Definition}
\newtheorem{remark}[thm]{Remark}
\newtheorem{example}[thm]{Example}
\newtheorem{assumption}{Assumption}

\makeatletter \@addtoreset{equation}{section} \makeatother

\renewcommand{\P}{\mathbb{P}}
\newcommand{\E}{\mathbb{E}}

\newcommand{\R}{\mathbb{R}}

\newcommand{\N}{\mathbb{N}}

\newcommand{\ep}{\varepsilon}

\renewcommand{\d}{\mathrm{d}}
\newcommand{\m}{\mathfrak{m} }

\newcommand{\DD}{{\Delta\hspace{-0.29cm}\Delta}^{\rm\! HK}}

\newcommand{\loc}{{\rm loc}}
\newcommand{\1}{{\bf 1}}
\newcommand{\eps}{{\varepsilon}}
\newcommand{\wt}{\widetilde}

\renewcommand{\wt}{\widetilde}

\title{\large\bf Hess-Schrader-Uhlenbrock inequality for the heat semigroup on differential forms over 
Dirichlet spaces tamed by distributional curvature lower bounds}
\author{
Kazuhiro Kuwae\footnote{\Letter Kazuhiro Kuwae ({\tt kuwae@fukuoka-u.ac.jp})
Department of Applied Mathematics, Fukuoka University,
Fukuoka 814-0180, Japan. 
Supported in part by JSPS Grant-in-Aid for Scientific Research (S) (No. 22H04942) and fund (No.~215001) from the Central Research Institute of Fukuoka University.}
}
\date{}
\begin{document}
\maketitle

\begin{abstract}
The notion of tamed Dirichlet space was proposed by Erbar, Rigoni, Sturm and Tamanini~\cite{ERST} as a Dirichlet space having a weak form of Bakry-\'Emery curvature lower bounds in distribution sense. 
After their work, Braun~\cite{Braun:Tamed2021} established  a vector calculus for it, in  particular, the space of $L^2$-normed $L^{\infty}$-module describing vector fields, 
$1$-forms, Hessian in $L^2$-sense. 
In this framework,    
we establish the Hess-Schrader-Uhlenbrock inequality for $1$-forms as an element of $L^2$-cotangent module  
(an $L^2$-normed $L^{\infty}$-module), which  extends
the Hess-Schrader-Uhlenbrock inequality by Braun~\cite{Braun:Tamed2021} under an additional condition.  
\end{abstract}


\section{Statement of Main Theorem}\label{sec:StatementMain}

\subsection{Framework}\label{subsec:Frame}
Let $(M,\tau)$ be a topological Lusin space, i.e., a continuous injective image of a Polish 
space, endowed with a $\sigma$-finite Borel measure $\m$ on $M$ with full topological support. 
Let $(\mathscr{E},D(\mathscr{E}))$ be a quasi-regular symmetric strongly local Dirichlet space on $L^2(M;\m)$ 
and $(P_t)_{t\geq0}$ (resp.~$(G_{\alpha})_{\alpha>0}$) the associated symmetric sub-Markovian strongly continuous semigroup (resp.~resolvent) on $L^2(M;\m)$ with the relation 
$G_{\alpha}f:=\int_0^{\infty}e^{-\alpha t}P_tf\d t$ for $f\in L^2(M;\m)$
(see \cite[Chapter IV, Definition~3]{MR} for the quasi-regularity and see \cite[Theorem~5.1(i)]{Kw:func} for the strong locality). 
Then there exists an $\m$-symmetric special standard process ${\bf X}=(\Omega, X_t, \P_x)$ 
associated with $(\mathscr{E},D(\mathscr{E}))$, i.e. for $f\in L^2(M;\m)\cap \mathscr{B}(M)$, 
$P_tf(x)=p_tf(x)$ $\m$-a.e.~$x\in M$ (see \cite[Chapter IV, Section~3]{MR}). Here $p_tf(x):=\E_x[f(X_t)]$ and $\mathscr{B}(M)$ denotes the family of Borel measurable functions on $M$ (the symbol $\mathscr{B}(M)$ is also used for the family of Borel measurable subsets of $M$ in some context). 
It is known that $(P_t)_{t\geq0}$ (resp.~$(G_{\alpha})_{\alpha>0}$) under $p\in[1,+\infty]$ can be extended to bounded contractive operators (still denoted by the same symbol) on $L^p(M;\m)$ and 
is strongly continuous on $L^p(M;\m)$ under $p\in[1,+\infty[$ in the sense that 
$\lim_{t\to0}\|P_tf-f\|_{L^p(M;\m)}=0$ (resp.~$\lim_{\alpha\to\infty}\|\alpha G_{\alpha}f-f\|_{L^p(M;\m)}=0$), and weakly* continuous on $L^{\infty}(M;\m)$ in the sense that $\lim_{t\to0}\int_M(P_tf-f)g\d\m=0$ 
(resp.~$\lim_{\alpha\to\infty}\int_M(\alpha G_{\alpha}f-f)g\d\m=0$) for $f\in L^{\infty}(M;\m)$ and $g\in L^1(M;\m)$. 
An increasing sequence $\{F_n\}$ of closed subsets is called an \emph{$\mathscr{E}$-nest} if 
$\bigcup_{n=1}^{\infty}D(\mathscr{E})_{F_n}$ is $\mathscr{E}_1$-dense in $D(\mathscr{E})$, where 
$D(\mathscr{E})_{F_n}:=\{u\in D(\mathscr{E})\mid u=0\;\m\text{-a.e.~on }M\setminus F_n\}$. 
A subset $N$ is said to be \emph{$\mathscr{E}$-exceptional} or \emph{$\mathscr{E}$-polar} if 
there exists an $\mathscr{E}$-nest $\{F_n\}$ such that $N\subset \bigcap_{n=1}^{\infty}(M\setminus F_n)$. 
For two subsets $A,B$ of $M$, we write $A\subset B$ $\mathscr{E}$-q.e. if $A\setminus B$ is $\mathscr{E}$-exceptional, hence we write $A=B$ $\mathscr{E}$-q.e. if $A\setminus B$ $\mathscr{E}$-q.e. and 
$B\setminus A$ $\mathscr{E}$-q.e. hold. 
It is known that any $\mathscr{E}$-exceptional set is $\m$-negligible and for any $\m$-negligible 
$\mathscr{E}$-quasi-open set is $\mathscr{E}$-exceptional, in particular, for an $\mathscr{E}$-quasi continuous function $u$, $u\geq0$ $\m$-a.e. implies $u\geq0$ $\mathscr{E}$-q.e.   
For a statement $P(x)$ with respect to $x\in M$, we say that $P(x)$ holds $\mathscr{E}$-q.e.~$x\in M$ (simply $P$ holds $\mathscr{E}$-q.e.) if the set $\{x\in M\mid P(x)\text{ holds }\}$ is $\mathscr{E}$-exceptional. 
A subset $G$ of $M$ is called an \emph{$\mathscr{E}$-quasi-open set} if there exists an $\mathscr{E}$-nest $\{F_n\}$ such that $G\cap F_n$ is an open set of $F_n$ with respect to the relative topology on 
$F_n$ for each $n\in\mathbb{N}$. A function $u$ on $M$ is said to be \emph{$\mathscr{E}$-quasi continuous} if there exists an $\mathscr{E}$-nest $\{F_n\}$ such that $u|_{F_n}$ is continuous on $F_n$ for each $n\in\mathbb{N}$.

Denote by $(\mathcal{E},D(\mathcal{E})_e)$ the extended Dirichlet space of $(\mathscr{E},D(\mathscr{E}))$ defined by 
\begin{align*}
\left\{\begin{array}{rl}D(\mathcal{E})_e&=\{u\in L^0(M;\m)\mid \exists \{u_n\}\subset D(\mathscr{E}): \mathscr{E}\text{-Cauchy sequence}\\
&\hspace{6cm}  \text{ such  that }\lim_{n\to\infty}u_n=u\;\m\text{-a.e.}\}, \\
\mathscr{E}(u,u)&=\lim_{n\to\infty}\mathscr{E}(u_n,u_n),\end{array}\right.
\end{align*}
(see \cite[p.~40]{FOT} for details on extended Dirichlet space). 
Denote by $\dot{D}(\mathscr{E})_{\loc}$, the space of functions locally in $D(\mathscr{E})$ in the broad sense defined 
by 
\begin{align*}
\dot{D}(\mathscr{E})_{\loc}&=\{u\in L^0(M;\m)\mid \exists \{u_n\}\subset D(\mathscr{E})\text{ and }\exists \{G_n\}:\mathscr{E}\text{-nest of }\mathscr{E}\text{-quasi-open sets}\\ 
&\hspace{8cm}  \text{ such  that }u=u_n\; \m\text{-a.e.~on }G_n\}.
\end{align*}
Here an increasing sequence $\{G_n\}$ of $\mathscr{E}$-quasi open sets is called an 
\emph{$\mathscr{E}$-nest} if $M=\bigcup_{n=1}^{\infty}G_n$ $\mathscr{E}$-q.e. 
For $u\in \dot{D}(\mathscr{E})_{\loc}$ and an $\mathscr{E}$-nest $\{G_n\}$ 
of $\mathscr{E}$-quasi-open sets satisfying $u|_{G_n}\in D(\mathscr{E})|_{G_n}$, we write $\{G_n\}\in \Xi(u)$. It is known that $D(\mathscr{E})\subset D(\mathscr{E})_e\subset \dot{D}(\mathscr{E})_{\loc}$ and 
each $u\in \dot{D}(\mathscr{E})_{\loc}$ admits an $\mathscr{E}$-quasi continuous $\m$-version (see \cite{Kw:func}).

It is known that for $u,v\in D(\mathscr{E})\cap L^{\infty}(M;\m)$ there exists a unique signed finite Borel 
measure $\mu_{\langle u,v\rangle }$ on $M$ such that 
\begin{align*}
2\int_M\tilde{f}\d\mu_{\langle u,v\rangle }=\mathscr{E}(uf,v)+\mathscr{E}(vf,u)-\mathscr{E}(uv,f)\quad \text{ for }\quad u,v\in D(\mathscr{E})\cap L^{\infty}(M;\m).
\end{align*}
Here $\tilde{f}$ denotes the $\mathscr{E}$-quasi-continuous $\m$-version of $f$ (see \cite[Theorem~2.1.3]{FOT}, \cite[Chapter IV, Proposition~3.3(ii)]{MR}). 
We set $\mu_{\langle f\rangle }:=\mu_{\langle f,f\rangle }$ for $f\in D(\mathscr{E})\cap L^{\infty}(M;\m)$. 
Moreover, 
for $f,g\in D(\mathscr{E})$, there exists a signed finite measure $\mu_{\langle f,g\rangle }$ on $M$ such that 
$\mathscr{E}(f,g)=\mu_{\langle f,g\rangle }(M)$, hence $\mathscr{E}(f,f)=\mu_{\langle f\rangle }(M)$. 
We assume $(\mathscr{E},D(\mathscr{E}))$ admits a carr\'e-du-champ $\Gamma$, i.e. 
$\mu_{\langle f\rangle }\ll\m$ for all $f\in D(\mathscr{E})$ and set $\Gamma(f):=\d\mu_{\langle f\rangle }/\d\m$. 
Then $\mu_{\langle f,g\rangle }\ll\m$ for all $f,g\in D(\mathscr{E})$ and $\Gamma(f,g):=
\d\mu_{\langle f,g\rangle }/{\d\m}\in L^1(M;\m)$ is expressed by $\Gamma(f,g)=\frac14(\Gamma(f+g)-\Gamma(f-g))$ for $f,g\in D(\mathscr{E})$. 

We can extend the singed smooth $\mu_{\langle f,g\rangle }$ or carr\'e-du-champ $\Gamma(f,g)$ from $f,g\in D(\mathscr{E})$ to 
$f,g\in \dot{D}(\mathscr{E})_{\loc}$ by polarization (see \cite[Lemmas~5.2 and 5.3]{Kw:func}). In this case, $\mu_{\langle f,g\rangle}$ (resp.~$\Gamma(f,g)$) 
is no longer a finite signed measure (resp.~an $\m$-integrable function). For $u\in \dot{D}(\mathscr{E})_{\loc}$ and 
an $\mathscr{E}$-nest $\{G_n\}$ 
of $\mathscr{E}$-quasi-open sets satisfying $u|_{G_n}\in D(\mathscr{E})|_{G_n}$, then we can define $\mathscr{E}(u,v)$ 
for $v\in\bigcup_{n=1}^{\infty}D(\mathscr{E})_{G_n}$ by 
\begin{align*}
\mathscr{E}(u,v)=\mu_{\langle u,v\rangle}(M).
\end{align*}
Here $D(\mathscr{E})_{G_n}:=\{u\in D(\mathscr{E})\mid \tilde{u}=0 \;\mathscr{E}\text{-q.e.~on }G_n^c\}$.

Fix $q\in\{1,2\}$. Let $\kappa$ be a signed smooth measure with its Jordan-Hahn decomposition $\kappa=\kappa^+-\kappa^-$. 
We assume that $\kappa^+$ is of Dynkin class smooth measure ($\kappa^+\in S_D({\bf X})$ in short) 
and $2\kappa^-$ is of extended Kato class smooth measure ($2\kappa^-\in S_{E\!K}({\bf X})$ in short). 
More precisely, $\nu\in S_D({\bf X})$ (resp.~$\nu\in S_{E\!K}({\bf X})$) if and only if $\nu\in S({\bf X})$ and 
$\m\text{-}\sup_{x\in M}\E_x[A_t^{\nu}]<\infty$ for any/some $t>0$ 
(resp.~$\lim_{t\to0}\m\text{-}\sup_{x\in M}\E_x[A_t^{\nu}]<1$) (see \cite{AM:AF}). 
Here $S({\bf X})$ denotes the family of smooth measures with respect to ${\bf X}$ (see \cite[Chapter VI, Definition~2.3]{MR}, \cite[p.~83]{FOT} for the definition of smooth measures) and $\m$-$\sup_{x\in M}f(x)$ denotes the $\m$-essentially supremum for a function $f$ on $M$.  
For $\nu\in  S({\bf X})$ and set $U_{\alpha}\nu(x):=\E_x\left[\int_0^{\infty}e^{-\alpha t}\d A_t^{\nu}\right]$ with its $\m$-essentially supremum $\|U_{\alpha}\nu\|_{\infty}:=\m$-$\sup_{x\in M}U_{\alpha}\nu(x)$, $\|U_{\alpha}\nu\|_{\infty}<\infty$ for some/any $\alpha>0$ 
(resp.~$\lim_{\alpha\to\infty}\|U_{\alpha}\nu\|_{\infty}=0$, $\lim_{\alpha\to\infty}\|U_{\alpha}\nu\|_{\infty}<1$)
if and only if $\nu\in S_D({\bf X})$ (resp.~$\nu\in  S_{E\!K}({\bf X})$, $\nu\in S_K({\bf X})$).  
For $\nu\in S_D({\bf X})$, the following inequality holds
\begin{align}
\int_M\tilde{f}^2\d\nu\leq\|U_{\alpha}\nu\|_{\infty}\mathscr{E}_{\alpha}(f,f),\label{eq:StollmannVoigt}
\end{align}
which is called the Stollmann-Voigt's inequality. Here $\mathscr{E}_{\alpha}(f,f):=
\mathscr{E}_{\alpha}(f,f)+\alpha\|f\|_{L^2(M;\m)}^2$.

Then, for $q\in\{1,2\}$, the quadratic form 
\begin{align*}
\mathscr{E}^{q\kappa}(f):=\mathscr{E}(f)+\langle 2\kappa, \tilde{f}^2\rangle 
\end{align*}
with finiteness domain $D(\mathscr{E}^{q\kappa})=D(\mathscr{E})$ is closed, lower semi-bounded, moreover, 
there exists $\alpha_0>0$ and $C>0$ such that 
\begin{align*}
C^{-1}\mathscr{E}_1(f,f)\leq \mathscr{E}^{q\kappa}_{\alpha_0}(f,f)\leq C\mathscr{E}_1(f,f)\quad \text{ for all }\quad f\in D(\mathscr{E}^{q\kappa})=D(\mathscr{E})
\end{align*}
by \eqref{eq:StollmannVoigt}
(see \cite[(3.3)]{CFKZ:Pert} and \cite[Assumption of Theorem~1.1]{CFKZ:GenPert}).   
The Feynman-Kac semigroup $(p_t^{q\kappa})_{t\geq0}$ defined by 
\begin{align*}
p_t^{q\kappa}f(x)=\E_x[e^{-qA_t^{\kappa}}f(X_t)],\quad f\in \mathscr{B}_b(M)
\end{align*}
is $\m$-symmetric, i.e. 
\begin{align*}
\int_M p_t^{q\kappa}f(x)g(x)\m(\d x)=\int_M g(x)p_t^{q\kappa}g(x)\m(\d x)\quad\text{ for all }\quad f,g\in \mathscr{B}_+(M)
\end{align*} 
and coincides with the strongly continuous semigroup $(P_t^{q\kappa})_{t\geq0}$ on $L^2(M;\m)$ associated with 
$(\mathscr{E}^{q\kappa}, D(\mathscr{E}^{q\kappa}))$ (see \cite[Theorem~1.1]{CFKZ:GenPert}). 
Here $A_t^{q\kappa}$ is a continuous additive functional (CAF in short) associated with the signed smooth measure $q\kappa$ under Revuz correspondence. 
Under $\kappa^{-}\in S_{K}({\bf X})$ and $p\in[1,+\infty]$, 
$(p_t^{\kappa})_{t\geq0}$ can be extended 
to be a bounded operator on $L^{p}(M;\m)$ denoted by  $P_t^{\kappa}$ such that 
there exist finite constants $C(\kappa)>0, C_{\kappa}\geq0$ depending only on $\kappa^-$ such that 
for every $t\geq0$
\begin{align}
\|P_t^{\kappa}\|_{p,p}\leq C(\kappa)e^{C_{\kappa}t}.\label{eq:KatoContraction}
\end{align}
Here $C(\kappa)=1$ under $\kappa^-=0$. 
$C_{\kappa}\geq0$  
can be taken to be $0$ under $\kappa^-=0$. We define $R_{\alpha}^{\kappa}$ on $L^p(M;\m)$ for 
$\alpha>C_{\kappa}$ by 
\begin{align*}
R_{\alpha}^{\kappa}f:=\int_0^{\infty}e^{-\alpha t}P_t^{\kappa}f\,\d t.
\end{align*}
Let $\Delta^{q\kappa}$  be the $L^2$-generator associated with $(\mathscr{E}^{q\kappa},D(\mathscr{E}))$ 
defined by   
\begin{align}
\left\{\begin{array}{ll} D(\Delta^{q\kappa})&:=\{u\in D(\mathscr{E})\mid \text{there exists } w\in L^2(M;\m)\text{ such that }\\
&\hspace{3cm}\mathscr{E}^{q\kappa}(u,v)=-\int_M wv\,\d\m\quad \text{ for any }\quad v\in D(\mathscr{E})\}, 
\\ \Delta^{q\kappa} u&:=w\quad\text{ for } w\in L^2(M;\m)\quad\text{specified as above,} \end{array}\right. \label{eq:generatorL2}
\end{align}
called the {\it Schr\"odinger operator} with potential $q\kappa$.  
Formally, $\Delta^{q\kappa}$ can be understood as \lq\lq$\Delta^{q\kappa}=\Delta-q\kappa$\rq\rq, 
where $\Delta$ is the $L^2$-generator associated with $(\mathscr{E},D(\mathscr{E}))$.

\begin{defn}[$q$-Bakry-\'Emery condition]
{\rm 
Suppose that $q\in\{1,2\}$, $\kappa^+\in S_D({\bf X})$, $2\kappa^-\in S_{E\!K}({\bf X})$ and $N\in[1,+\infty]$. We say that $(M,\mathscr{E},\m)$ or simply $M$ satisfies the $q$-Bakry-\'Emery condition, briefly 
${\sf BE}_q(\kappa,N)$, if for every $f\in  D(\Delta)$ with $\Delta f\in D(\mathscr{E})$ 
and every nonnegative $\phi\in D(\Delta^{q\kappa})$ with 
$\Delta^{q\kappa}\phi\in L^{\infty}(M;\m)$ (we further impose $\phi\in L^{\infty}(M;\m)$ for $q=2$), we have 
\begin{align*}
\frac{1}{q}\int_M \Delta^{q\kappa}\phi \Gamma(f)^{\frac{q}{2}}\d\m-\int_M \phi
\Gamma(f)^{\frac{q-2}{2}}
\Gamma(f,\Delta f)\d\m
\geq \frac{1}{N}\int_M \phi\Gamma(f)^{\frac{q-2}{2}}(\Delta f)^2\d\m.
\end{align*}
The latter term is understood as $0$ if $N=\infty$.
}
\end{defn}

\begin{assumption}\label{asmp:Tamed}
{\rm We assume that $M$ satisfies ${\sf BE}_{2}(\kappa,N)$ condition for a given signed smooth measure  
$\kappa$ with $\kappa^+\in S_D({\bf X})$ and $2\kappa^-\in S_{E\!K}({\bf X})$ and $N\in[1,\infty]$.
}
\end{assumption}
Under Assumption~\ref{asmp:Tamed}, we say that $(M,\mathscr{E},\m)$ or simply $M$ is {\it tamed}. 
In fact, under $\kappa^+\in S_D({\bf X})$ and $2\kappa^-\in S_{E\!K}({\bf X})$, the condition ${\sf BE}_2(\kappa,\infty)$ is {\it equivalent} 
to ${\sf BE}_1(\kappa,\infty)$ (see Lemma~\ref{lem:BakryEmeryEquivalence} below), in particular, the heat flow $(P_t)_{t\geq0}$ 
satisfies 
\begin{align}
\sqrt{\Gamma(P_tf)}\leq P_t^{\kappa}\sqrt{\Gamma(f)}\quad\m\text{-a.e.~for any }f\in D(\mathscr{E})\quad \text{ and }\quad t\geq0\label{eq:gradCont}
\end{align}
(see \cite[Definition~3.3 and Theorem~3.4]{ERST}).
The inequality \eqref{eq:gradCont} plays a crucial role in our paper. 
Note that our condition $\kappa^+\in S_D({\bf X})$, $\kappa^-\in S_{E\!K}({\bf X})$ (resp.~$\kappa^+\in S_D({\bf X})$, $2\kappa^-\in S_{E\!K}({\bf X})$) is stronger than the $1$-moderate (resp.~$2$-moderate) condition treated in \cite{ERST} for the definition of tamed space.

\bigskip

Denote by $L^2(T^*\!M)$ the space of $L^2$-cotangent module defined as an $L^2$-normed $L^{\infty}$-module 
in Definition~\ref{df:CotangentModule} below,  
and by $W^{1,2}(T^*\!M)$ the Sobolev space of $L^2$-$1$-forms defined by 
$W^{1,2}(T^*\!M):=D({\rm d})\cap D({\rm d}_*)$, where $D({\rm d})$ (resp.~$D({\rm d}_*)$) is the domain of exterior differential (resp.~codifferential) defined in Definition~\ref{def:d} (resp.~Definition~\ref{def:d*}) below.  
We endow $W^{1,2}(T^*\!M)$ with the norm $\|\cdot\|_{W^{1,2}(T^*\!M)}$ given by 
\begin{align*}
\|\omega\|_{W^{1,2}(T^*\!M)}^2:=\|\omega\|_{L^2(T^*\!M)}^2+\|\d \omega\|_{L^2(\Lambda^{2}T^*\!M)}^2+
\|\d_*\omega\|_{L^2(M;\m)}^2.
\end{align*}
Here $L^2(\Lambda^{2}T^*\!M):=\Lambda^2 L^2(T^*\!M)$ denotes the space of $2$-fold exterior products 
of $L^2(T^*\!M)$ defined in Definition~\ref{def:exteriorproducts}. 
We define the \emph{contravariant} functional $\mathscr{E}_{\rm con}: L^2(T^*\!M)\to[0,+\infty]$ by 
\begin{align*}
\mathscr{E}_{\rm con}(\omega):=\left\{\begin{array}{cc}\displaystyle{\int_M\left[|\d\omega|^2+|\d_*\omega|^2 \right]\d\m} & \text{ if }\omega\in W^{1,2}(T^*\!M), \\\infty & \text{otherwise.} \end{array}\right.
\end{align*}
We define 
the subspace $H^{1,2}(T^*\!M)\subset W^{1,2}(T^*\!M)$ by 
the closure of ${\rm Reg}(T^*\!M)$ with respect to $\|\cdot\|_{W^{1,2}(T^*\!M)}$: 
\begin{align*}
H^{1,2}(T^*\!M):=\overline{{\rm Reg}(T^*\!M)}^{\|\cdot\|_{W^{1,2}(T^*\!M)}}.
\end{align*}
Here ${\rm Reg}(T^*\!M)$ is the space of regular $1$-forms defined in Definition~\ref{def:Test1Forms} below. 
The space $D(\DD)$ is defined to consist of all $\omega\in H^{1,2}(T^*\!M)$ for which there exists $\alpha\in L^2(T^*\!M)$ such that for every $\eta\in  H^{1,2}(T^*\!M)$,
\begin{align*}
\int_M \langle \alpha,\eta\rangle \d\m=-\int_M \left[\langle \d\omega,\d\eta\rangle +\langle \d_*\omega,\d_*\eta\rangle  \right]\d\m.
\end{align*}
In case of existence, the element $\alpha$ is unique, denoted by $\DD\omega$ and called the 
\emph{Hodge Laplacian}, \emph{Hodge-Kodaira Laplacian} or \emph{Hodge-de~\!\!Rham Laplacian} of $\omega$. 
Formally $\DD\omega$ can  be written \lq\lq$\DD\omega=-(\d\d_*+\d_*\d)\omega$\rq\rq.  
 
Let $P_t^{\rm HK}$ be the heat flow on $1$-forms associated to the functional $\widetilde{\mathscr{E}}_{\rm con}:L^2(T^*\!M)\to [0,+\infty]$ with 
\begin{align*}
\widetilde{\mathscr{E}}_{\rm con}(\omega):=\left\{\begin{array}{cc}\displaystyle{\int_M \left[|\d\omega|^2+|\d_*\omega|^2 \right]\d\m} & \text{ if }\omega\in H^{1,2}(T^*\!M), \\ \infty & \text{ otherwise.}\end{array}\right.
\end{align*}
We can write $\mathscr{E}^{\rm HK}$ instead of $\widetilde{\mathscr{E}}_{\rm con}$. Let $(P_t^{\rm HK})_{t\geq0}$ be the heat semigroup of bounded linear and self-adjoint operator on $L^2(T^*\!M)$ formally written by 
\begin{align*}
\text{\lq\lq $P_t^{\rm HK}:=e^{t\,\DD}$\rq\rq}.
\end{align*}
We define $R_{\alpha}^{\rm HK}\omega$ for $\omega\in L^2(T^*\!M)$ by 
\begin{align*}
R_{\alpha}^{\rm HK}\omega:=\int_0^{\infty}e^{-\alpha t}P_t^{\rm HK}\omega \,\d t.
\end{align*}

\subsection{Main Results}\label{subsec:Main}

Our main theorem under Assumption~\ref{asmp:Tamed} is the following: 

\begin{thm}[{Hess-Schrader-Uhlenbrock inequality}]\label{thm:HessShcraderUhlenbrock}
We have the following: 
\begin{enumerate}
\item For every $\omega\in L^2(T^*\!M)$ and $\alpha>C_{\kappa}$, 
\begin{align}
|R_{\alpha}^{\rm HK}\omega|\leq R_{\alpha}^{\kappa}|\omega|\quad\m\text{-a.e.}\label{eq:ResolventHSU}
\end{align}
\item For every $\omega\in L^2(T^*\!M)$ and every $t\geq0$, 
\begin{align}
|P_t^{\rm HK}\omega|\leq P_t^{\kappa}|\omega|\quad\m\text{-a.e.}\label{eq:HSU}
\end{align}
\end{enumerate}
\end{thm}

The assertion of Theorem~\ref{thm:HessShcraderUhlenbrock}(ii) is known as 
Hess-Schrader-Uhlenbrock inequality for $L^2$-differential forms, which was firstly established in the case that $M$ is a compact Riemannian manifold without boundary by Hess-Schrader-Uhlenbrock~\cite{HSU*, HSU} and by Simon~\cite{Simon}. 
This inequality was extended to the case of compact Riemannian manifold with convex boundary by Ouhabaz~\cite{Ouhabaz} and by Shigekawa~\cite{ShigekawaIntertwining,ShigekawaSemigroupDomination}.  
For general compact Riemannian manifolds with boundary, Hess-Schrader-Uhlenbrock inequality 
was proved by Hsu~\cite{Hsu:2002} in probabilistic way. 
For non-compact Riemannian manifolds without boundary, it was proved by G\"uneysu~\cite{Guneysu} in analytic way, and by Driver-Thalmaier~\cite{DriverTalmaier} and Elworthy-Le Jan-Li~\cite{ElworthyLeJanLi,XMLi} in stochastic way. $L^p$-properties of $(P_t^{\rm HK})_{t\geq0}$ and related heat kernel estimates on Riemannian manifolds have been studied by Magniez-Ouhabaz~\cite{MagniezOuhabaz} under the Ricci lower bounds having Kato class condition. 
Very recently, Theorem~\ref{thm:HessShcraderUhlenbrock}(ii) is proved by Braun~\cite[Theorem~A]{Braun:HeatFlow} in the framework of RCD$(K,\infty)$-space for $K\in\R$, and also 
proved by Braun~\cite[Theorem~8.41]{Braun:Tamed2021} in the present framework with an extra assumption \cite[Assumption~8.37]{Braun:Tamed2021}. Our Theorem~\ref{thm:HessShcraderUhlenbrock}(ii) extends 
\cite[Theorem~8.41]{Braun:Tamed2021} completely. 
Theorem~\ref{thm:HessShcraderUhlenbrock}(ii) also plays a crucial role to prove the $L^p$-boundedness of Riesz operator for $p\in[2,+\infty[$ (see \cite{EXKRiesz}).  

\bigskip

As a corollary of Theorem~\ref{thm:HessShcraderUhlenbrock}, we can establish the strong continuity of 
the semigroup $(P_t^{\rm HK})_{t\geq0}$ on $L^p(T^*\!M)$ under Kato class Ricci lower bounds. 

\begin{cor}[{$C_0$-property of $(P_t^{\rm HK})_{t\geq0}$ on $L^p(M;\m)$}]\label{cor:HessShcraderUhlenbrock}
 Suppose $p\in[2,+\infty]$, or $\kappa^-\in S_K({\bf X})$ and $p\in[1,+\infty]$.  
Then the heat flow $(P_t^{\rm HK})_{t\geq0}$ can be extended to a semigroup on $L^p(T^*\!M)$ and and 
for each $t>0$
\begin{align}
\|P_t^{\rm HK}\omega\|_{L^p(T^*\!M)}\leq 
C(\kappa)e^{C_{\kappa}t}\|\omega\|_{L^p(T^*\!M)},\quad \omega\in L^p(T^*\!M).\label{eq:LpContHKHF}
\end{align}
Moreover, if $\kappa^-\in S_K({\bf X})$ and $p\in[1,+\infty[$, then  
$(P_t^{\rm HK})_{t\geq0}$ is strongly continuous on $L^p(T^*\!M)$, i.e., 
$(P_t^{\rm HK})_{t\geq0}$ is a $C_0$-semigroup on $L^p(T^*\!M)$, and further  
$(P_t^{\rm HK})_{t\geq0}$ is weakly* continuous on $L^{\infty}(T^*\!M)$.
\end{cor}

\section{Smooth measures of extended Kato class}\label{sec:ExtendeKato}
In this section, we summarize results on smooth measures of extended Kato class. 
Denote by $S({\bf X})$ the family of (non-negative) smooth measures on $(M,\mathscr{B}(M))$ (see \cite[p.~83]{FOT} for the precise 
definition of smooth measures). Then there exists a PCAF $A^{\nu}$ admitting exceptional set satisfying 
\begin{align*}
\E_{h\m}\left[\int_0^{\infty}e^{-\alpha t}f(X_t)\d A_t^{\nu} \right]=\E_{f\nu}\left[\int_0^{\infty}e^{-\alpha t}f(X_t)\d t \right],
\quad f,h\in\mathscr{B}_+(M)
\end{align*}  
(see \cite[(5.14)]{FOT}). 
A measure $\nu\in S({\bf X})$ is said to be of \emph{Dynkin class} (resp.~\emph{Kato class}) if $\|\E_{\cdot}[A_t^{\nu}]\|_{L^{\infty}(M;\m)}<\infty$ for some/all $t>0$ (resp.~$\|\E_{\cdot}[A_t^{\nu}]\|_{L^{\infty}(M;\m)}=0$). 
A measure $\nu\in S({\bf X})$ is said to be of \emph{extended Kato class} if $\lim_{t\to0}\|\E_x[A_t^{\nu}]\|_{L^{\infty}(M;\m)}<1$. Denote by $S_D({\bf X})$ (resp.~$S_K({\bf X})$, $S_{E\!K}({\bf X})$) the family of 
(non-negative)  smooth measures of Dynkin class (resp.~Kato class, extended Kato class). For $\nu\in S_D({\bf X})$, 
$x\mapsto \E_x[A_t^{\nu}]$ is $\mathscr{E}$-quasi-continuous, hence $\|\E_{\cdot}[A_t^{\nu}]\|_{L^{\infty}(M;\m)}$ actually coincides with a quasi-essentially supremum: 
\begin{align*}
\|\E_{\cdot}[A_t^{\nu}]\|_{L^{\infty}(M;\m)}=q\text{-}\sup_{x\in M}\E_x[A_t^{\nu}]:=\inf_{N:\mathscr{E}\text{-exceptional}}\sup_{x\in N^c}\E_x[A_t^{\nu}].
\end{align*}
For $\nu\in S_{E\!K}({\bf X})$, there exist positive constants $C=C(\nu)$ and $C_{\nu}>0$ such that 
\begin{align}
\|\E_{\cdot}[e^{A_t^{\nu}}]\|_{L^{\infty}(M;\m)}\leq Ce^{C_{\nu}t}.\label{eq:unifKato}
\end{align}
Indeed, we can choose $C(\nu):=(1-\|\E_{\cdot}[A_{t_0}^{\nu}]\|_{L^{\infty}(M;\m)})^{-1}$ and $C_{\nu}:=\frac{1}{t_0}\log C(\nu)$. 
for some sufficiently small $t_0\in]0,1[$ satisfying $\|\E_{\cdot}[A_{t_0}^{\nu}]\|_{L^{\infty}(M;\m)}<1$ (see \cite[Proof of Theorem~2.2]{Sznitzman}). Denote by $P_t^{-\nu}f(x)$ the Feynman-Kac semigroup defined by 
$P_t^{-\nu}f(x):=\E_x[e^{A_t^{\nu}}f(X_t)]$, $f\in \mathscr{B}_+(M)$. \eqref{eq:unifKato} is nothing but 
$\|P_t^{-\nu}\|_{\infty,\infty}\leq Ce^{C_{\nu}t}$. Moreover, for the conjugate exponent 
$q:=p/(p-1)\in[1,+\infty]$ we can deduce 
\begin{align}
\|P_t^{-\nu}\|_{p,p}\leq C(q\nu)e^{\frac{C_{q\nu}}{q}t}\quad\text{ for }\quad p\in]1,+\infty].\label{eq:pKatoEst}
\end{align} 
Here $\|P_t^{-\nu}\|_{p,p}:=\sup\{\|P_t^{-\nu}f\|_{L^p(M;\m)}\mid f\in L^p(M;\m), \|f\|_{L^p(M;\m)}=1\}$ stands for the operator norm from 
$L^p(M;\m)$ to $L^p(M;\m)$. 
For each $\beta\geq1$ satisfying $\beta\nu\in S_{E\!K}({\bf X})$, we can choose
\begin{align}
C_{\beta\nu}\in]0,C_{\nu}]\quad\text{ and }\quad C(\nu)=C(\beta\nu).\label{eq:KatoCoincidence}
\end{align}
Indeed $C_{\nu}$ and $C(\nu)$ can be given by 
\begin{align}
C_{\nu}=\frac{1}{t_0}\log\left(1-\|\E_{\cdot}[A_{t_0}^{\nu}]\|_{\infty} \right)^{-1},\quad 
C(\nu)=\left(1-\|\E_{\cdot}[A_{t_0}^{\nu}]\|_{\infty} \right)^{-1}\label{eq:KatoExpression}
\end{align}
for some small $t_0\in]0,1]$ satisfying $\|\E_{\cdot}[A_{t_0}^{\nu}]\|_{\infty}<1$ (see \cite[Proof of Theorem~2.2]{Sznitzman}). 
If we choose $s_0\in]0,t_0]$ 
satisfying $\|\E_{\cdot}[A_{s_0}^{\nu}]\|_{\infty}=\|\E_{\cdot}[A_{t_0}^{\beta\nu}]\|_{\infty}<1$, then we see $C_{\beta\nu}\leq C_{\nu}$ and $C(\nu)=C(\beta\nu)$ for such selection of $s_0$ (resp.~$t_0$) to $C_{\nu}$ (resp.~$C_{\beta\kappa}$). 
If $\beta\nu\in S_{E\!K}({\bf X})$ for $\beta\geq1$, then we can deduce from \eqref{eq:pKatoEst} 
and \eqref{eq:KatoCoincidence} that 
for $p\in[\beta/(\beta-1),+\infty]$
\begin{align}
\|P_t^{-\nu}\|_{p,p}\leq C(\nu)e^{C_{\nu}t}.\label{eq:generalpKatoEst}
\end{align}
In particular, if $\nu\in S_K({\bf X})$, then \eqref{eq:generalpKatoEst} holds for any $p\in[1,+\infty]$ 
(the proof for $p=1$ can be directly deduced). 
When $\nu=-R\m$ for a constant $R\in \R$, $C_{\nu}$ (resp.~$C(\nu)$) can be taken to be $R\lor 0$ (resp.~$1$).

It is known that $\nu\in S_D({\bf X})$ (resp.~$\nu\in S_K({\bf X})$) if and only if $\|\E_{\cdot}\left[\int_0^{\infty}e^{-\alpha t}\d A_t^{\nu}\right]\|_{L^{\infty}(M;\m)}<\infty$ for some/all $\alpha>0$ (resp.~$\lim_{\alpha\to\infty}\|\E_{\cdot}\left[\int_0^{\infty}e^{-\alpha t}\d A_t^{\nu} \right]\|_{L^{\infty}(M;\m)}=0$), and $\nu\in S_{E\!K}({\bf X})$ if and only if 
$\lim_{\alpha\to\infty}\|\E_{\cdot}\left[\int_0^{\infty}e^{-\alpha t}\d A_t^{\nu} \right]\|_{L^{\infty}(M;\m)}<1$. 

\section{Test functions}

Under Assumption~\ref{asmp:Tamed}, we now introduce ${\rm Test}(M)$ the set of test functions: 

\begin{defn}\label{def:TestFunc}
{\rm 
Let $(M,\mathscr{E},\m)$ be a tamed space. Let us define the set of {\it test functions} by 
\begin{align*}
{\rm Test}(M):&=\{f\in D(\Delta)\cap L^{\infty}(M;\m)\mid \Gamma(f)\in L^{\infty}(M;\m),\Delta f\in D(\mathscr{E})\},\\
{\rm Test}(M)_{fs}:&=\{f\in {\rm Test}(M)\mid \m({\rm supp}[f])<\infty\}.
\end{align*}
It is easy to see that ${\rm Test}(M)_{fs}\subset \mathscr{D}_{1,p}\subset H^{1,p}(M)$ for any $p\in]1,+\infty[$.
}
\end{defn}

The following lemmas hold. 
\begin{lem}[{\cite[Proposition~6.8]{ERST}}]\label{lem:boundedEst}
Under Assumption~\ref{asmp:Tamed}, ${\sf BE}_2(-\kappa^-,+\infty)$ holds. Moreover, for every $f\in L^2(M;\m)\cap L^{\infty}(M;\m)$ and $t>0$, it holds 
\begin{align}
\Gamma(P_tf)\leq\frac{1}{2t}\|P_t^{-2\kappa^-}\|_{\infty,\infty}\cdot\|f\|_{L^{\infty}(M;\m)}^2.\label{eq:BoundedEst}
\end{align}
In particular, for $f\in L^2(M;\m)\cap L^{\infty}(M;\m)$, then $P_tf\in {\rm Test}(M)$. 
\end{lem}

\begin{lem}[{\cite[Theorems~3.4 and 3.6, Proposition~3.7 and Theorem~6.10]{ERST}}]\label{lem:BakryEmeryEquivalence}\quad\\
Under $\kappa^+\in S_D({\bf X})$ and $2\kappa^-\in S_{E\!K}({\bf X})$. 
the condition ${\sf BE}_2(\kappa,\infty)$ is equivalent 
to ${\sf BE}_1(\kappa,\infty)$. In particular, we have \eqref{eq:gradCont}. 
\end{lem}

\begin{lem}[{\cite[Lemma~3.2]{Sav14}}]\label{lem:algebra}
Under Assumption~\ref{asmp:Tamed}, for every $f\in {\rm Test}(M)$, we have 
$\Gamma(f)\in D(\mathscr{E})\cap L^{\infty}(M;\m)$ and there exists $\mu=\mu^+-\mu^-$ with 
$\mu^{\pm}\in D(\mathscr{E})^*$ such that 
\begin{align}
-\mathscr{E}^{2\kappa}(u,\varphi)=\int_M \tilde{\varphi}\,\d \mu\quad\text{ for all }\quad \varphi\in D(\mathscr{E}).
\end{align}
Moreover, ${\rm Test}(M)$ is an algebra, i.e., for $f,g\in {\rm Test}(M)$, $fg\in {\rm Test}(M)$, 
if further {\boldmath$f$}$\in {\rm Test}(M)^n$, then $\Phi(${\boldmath$f$}$)\in {\rm Test}(M)$ for every smooth 
function $\Phi:\R^n\to\R$ with $\Phi(0)=0$.
\end{lem}

\begin{lem}\label{lem:DensenessTestFunc}
${\rm Test}(M)\cap L^p(M;\m)$ is dense in $L^p(M;\m)\cap L^2(M;\m)$ both in $L^p$-norm 
and in $L^2$-norm. In particular, ${\rm Test}(M)$ is dense in $(\mathscr{E},D(\mathscr{E}))$. 
\end{lem}
\begin{proof}[{\bf Proof}]
Take $f\in L^p(M;\m)\cap L^2(M;\m)$. We may assume $f\in L^{\infty}(M;\m)$, because 
$f$ is $L^p$(and also $L^2$)-approximated by a sequence $\{f^k\}$ of 
$L^p(M;\m)\cap L^2(M;\m)\cap L^{\infty}(M;\m)$-functions defined 
by $f^k:=(-k)\lor f\land k$. If $f\in L^p(M;\m)\cap L^2(M;\m)\cap L^{\infty}(M;\m)$, 
$P_tf\in {\rm Test}(M)$ by Lemma~\ref{lem:boundedEst} and 
$\{P_tf\}\subset {\rm Test}(M)\cap L^p(M;\m)$ converges to 
$f$ in $L^p$ and in $L^2$ as $t\to0$. If $f\in D(\mathscr{E}))$, then $f$ can be approximated by 
$\{P_tf^k\}$ in $(\mathscr{E},D(\mathscr{E}))$. This shows the last statement.
\end{proof}
\begin{remark}\label{rem:TestFunction}
{\rm 
As proved above, ${\rm Test}(M)$ forms an algebra and dense in $(\mathscr{E},D(\mathscr{E}))$ under 
Assumption~\ref{asmp:Tamed}. However, ${\rm Test}(M)$ is not necessarily a subspace of $C_b(M)$. 
When the tamed space comes from ${\rm RCD}$-space, the Sobolev-to-Lipschitz property of RCD-spaces ensures 
${\rm Test}(M)\subset C_b(M)$. 
}
\end{remark}

\section{Vector space calculus for tamed space}
\subsection{$L^{\infty}$-module}\label{subsec:normedModule}

We need the notion of $L^p$-normed $L^{\infty}(M;\m)$-modules. 
\begin{defn}[$L^{\infty}$-module]\label{df:NormedModule}
{\rm 
Given $p\in[1,+\infty]$, a real Banach space $(\mathscr{M},\|\cdot\|_{\mathscr{M}})$, or simply, $\mathscr{M}$ is called an \emph{$L^p$-normed $L^{\infty}$-module} (over $(M,\m)$) if it satisfies 
\begin{enumerate}
\item[(a)] a bilinear map $\cdot$ : $L^{\infty}(M;\m)\times\mathscr{M}\to\mathscr{M}$ satisfying 
\begin{align*}
(fg)\cdot v&=f\cdot(gv),\\
\1_M \cdot v&=v,
\end{align*} 
\item[(b)] a nonnegatively valued map $|\cdot|_{\m}:\mathscr{M}\to L^p(M;\m)$ such that 
\begin{align*}
|f\cdot v|_{\m}&=|f||v|_{\m}\quad \m\text{-a.e.},\\
\|v\|_{\mathscr{M}}&=\||v|_{\m}\|_{L^p(M;\m)},
\end{align*}
for every $f,g\in L^{\infty}(M;\m)$ and $v\in\mathscr{M}$. If only (a) is satisfied, we call 
$(\mathscr{M},\|\cdot\|_{\mathscr{M}})$ or simply $\mathscr{M}$ an $L^{\infty}(M;\m)$-module. 
\end{enumerate} 
Throughout this paper, we always assume that for every $v\in \mathscr{M}$ its point-wise norm $|v|_{\m}$ is Borel measurable.  
We call $\mathscr{M}$ an \emph{$L^{\infty}$-module} if it is  $L^p$-normed 
for some $p\in[1,+\infty]$. $\mathscr{M}$ is called \emph{separable} 
 if it is a separable Banach space. We call $v\in\mathscr{M}$ is 
($\m$-essentially) bounded if $|v|_{\m}\in L^{\infty}(M;\m)$. $\mathscr{M}$ is called \emph{Hilbert module} if is an $L^2$-normed $L^{\infty}$-module, in this case, the point-wise norm $|\cdot|_{\m}$ satisfies a point-wise $\m$-a.e.~parallelogram identity, hence it induces a \emph{point-wise scalar product} $\langle \cdot,\cdot\rangle_{\m}: \mathscr{M}\times \mathscr{M}\to L^1(M;\m)$ which is $L^{\infty}$-bilinear, $\m$-a.e.~nonnegative definite, local in both components, satisfies the point-wise $\m$-a.e.~Cauchy-Schwarz inequality, and reproduces the Hilbertian scalar product on $\mathscr{M}$ by integration with respect to $\m$.
}
\end{defn}
\begin{defn}[Dual module]\label{df:DualModule}
{\rm Let $\mathscr{M}$ and $\mathscr{N}$ be $L^p$-normed $L^{\infty}$-module. Denote both point-wise norms by $|\cdot|_{\m}$. A map $T:\mathscr{M}\to\mathscr{N}$ is called \emph{module morphism} 
if it is a bounded linear map in the sense of functional analysis and 
\begin{align}
T(f\,v)=f\,T(v)\label{eq:morphism}
\end{align}
for every $v\in\mathscr{M}$ and every $f\in L^{\infty}(M;\m)$. The set of all module morphisms 
is written ${\rm Hom}(\mathscr{M},\mathscr{N})$ and equipped with the usual operator norm $\|\cdot\|_{\mathscr{M},\mathscr{N}}$. We call $\mathscr{M}$ and $\mathscr{N}$ \emph{isomorphic} (as $L^{\infty}$-module) if there exists  $T\in{\rm Hom}(\mathscr{M},\mathscr{N})$ and $S\in {\rm Hom}(\mathscr{N},\mathscr{M})$ such that $T\circ S={\rm Id}_{\mathscr{N}}$ and $S\circ T={\rm Id}_{\mathscr{M}}$. 
Any such $T$ is called \emph{module isomorphism}. If  in addition, such a $T$ is a norm isometry, it is called \emph{module isometric isomorphism}. In fact, by \eqref{eq:morphism} every module isometric isomorphism $T$ preserves pointwise norms $\m$-a.e., i.e. for every $v\in\mathscr{M}$, 
\begin{align*}
|T(v)|_{\m}=|v|_{\m}\quad \m\text{-a.e.}
\end{align*}
The dual module $\mathscr{M}^*$ is defined by 
\begin{align*}
\mathscr{M}^*:={\rm Hom}(\mathscr{M}, L^1(M;\m))
\end{align*}
and will be endowed with the usual operator norm. The point-wise paring between 
$v\in\mathscr{M}$ and $L\in\mathscr{M}^*$ is denoted by $L(v)\in L^1(M;\m)$. If $\mathscr{M}$ is $L^p$-normed, then $\mathscr{M}^*$ is an $L^q$-normed $L^{\infty}$-module, where $p,q\in[1,+\infty]$ with $1/p+1/q=1$ (see \cite[Proposition~1.2.14]{Gigli:NonSmoothDifferentialStr}) with naturally defined multiplications and, and by a slight abuse of notation, point-wise norm is given by 
\begin{align}
|L|_{\m}:=\m\text{\rm-esssup}\{|L(v)|\mid v\in\mathscr{M}, |v|_{\m}\leq1\; \m\text{-a.e.}\}.\label{eq:DualPointNorm}
\end{align}
By \cite[Corollary~1.2.16]{Gigli:NonSmoothDifferentialStr}, if $p<\infty$,
\begin{align*}
|v|_{\m}=\m\text{\rm-esssup}\{|L(v)|\mid L\in\mathscr{M}, |L|_{\m}\leq1\; \m\text{-a.e.}\}
\end{align*}
for every $v\in\mathscr{M}$. Moreover, under $p<\infty$, in the sense of functional analysis 
$\mathscr{M}^*$ and the dual Banach space $\mathscr{M}'$ of $\mathscr{M}$ are isometrically isomorphic 
\cite[Proposition~1.2.13]{Gigli:NonSmoothDifferentialStr}. In this case, the natural point-wise paring map $\mathscr{I}:\mathscr{M}\to\mathscr{M}^{**}$, where 
$\mathscr{M}^{**}:={\rm Hom}(\mathscr{M}^*,L^1(M;\m))$, belongs to ${\rm Hom}(\mathscr{M},\mathscr{M}^{**})$ and constitutes a norm isometry is $L^p$-normed for $p\in]1,+\infty[$, this is equivalent to $\mathscr{M}$ being reflexive as Banach space \cite[Corollary~1.2.18]{Gigli:NonSmoothDifferentialStr}, 
while for $p=1$, the implication from \lq\lq reflexive as Banach space\rq\rq\, to \lq\lq reflexive as $L^{\infty}$-module\rq\rq\, still holds \cite[Propositions~1.2.13 and 1.2.17]{Gigli:NonSmoothDifferentialStr}. In particular all Hilbert modules are reflexive in both senses. 
}
\end{defn}

\begin{defn}[$L^0$-module]\label{df:L0Module}
{\rm Fix an $L^{\infty}$-module $\mathscr{M}$. Let $(B_i)_{i\in\mathbb{N}}$ a Borel partition of $M$ such that $\m(B_i)\in]0,+\infty[$ for each $i\in\mathbb{N}$. Denote by $\mathscr{M}^0$ the completion of  
$\mathscr{M}$ with respect to $\mathscr{M}$ the distance ${\sf d}_{\mathscr{M}^0}:\mathscr{M}\times\mathscr{M}\to[0,+\infty[$ defined by 
\begin{align*}
{\sf d}_{\mathscr{M}^0}(v,w):=\sum_{i\in\mathbb{N}}\frac{1}{2^i\m(B_i)}\int_{B_i}(|v-w|_{\m}\land 1)\d\m.
\end{align*}
We call $\mathscr{M}^0$ as the \emph{$L^0$-module} associated with $\mathscr{M}$. The induced topology on $\mathscr{M}^0$ does not depend on  the choice of $(B_i)_{i\in\mathbb{N}}$ (see \cite[p.~31]{Gigli:NonSmoothDifferentialStr}). In addition, scalar and functional multiplication, and the point-wise norm $|\cdot|_{\m}$ extend continuously to $\mathscr{M}^0$, so that all $\m$-a.e.~properties mentioned for $L^{\infty}$-module hold for general elements in $\mathscr{M}^0$ and $L^0(M;\m)$ in place of $\mathscr{M}$ and $L^{\infty}(M;\m)$. The point-wise paring of $\mathscr{M}$ and $\mathscr{M}^*$extends uniquely and continuously to a bilinear map on $\mathscr{M}\times (\mathscr{M}^*)^0$ with values in $L^0(M;\m)$ such that for every $v\in\mathscr{M}^0$ and every $L\in(\mathscr{M}^*)^0$, 
\begin{align*}
|L(v)|\leq |L|_{\m}|v|_{\m}\quad\m\text{-a.e.}
\end{align*} 
}
\end{defn}
We have the following characterization of elements in $(\mathscr{M}^*)^0$:
\begin{prop}[{\cite[Proposition~1.3.2]{Gigli:NonSmoothDifferentialStr}}]\label{prop:L0characterization}
Let $T:\mathscr{M}^0\to L^0(M;\m)$ be a linear map for which there exists $f\in L^0(M;\m)$ such that 
for every $v\in\mathscr{M}$, 
\begin{align*}
|T(v)|\leq f|v|_{\m}\quad\m\text{-a.e.}
\end{align*}
Then there exists a unique $L\in(\mathscr{M}^*)^0$ such that for every $v\in\mathscr{M}$, 
\begin{align*}
L(v)=T(v)\quad\m\text{-a.e.,}
\end{align*}
and we furthermore have
\begin{align*}
|L|_{\m}\leq f\quad\m\text{-a.e.}
\end{align*}
\end{prop}
\subsection{Tensor products}\label{subsec:TensorProducts}
\begin{defn}[Tensor products]\label{df:TensorProducts}
{\rm Let $\mathscr{M}_1$ and $\mathscr{M}_2$ be two Hilbert module. 
By a slight abuse of notation, 
we denote both point-wise scalar products by $\langle \cdot,\cdot\rangle_{\m}$. 
We consider the $L^0$-module $\mathscr{M}_i^0$ induced from $\mathscr{M}_i$, $i=1,2$. 
Let $\mathscr{M}_1^0\odot\mathscr{M}_2^0$ be the \lq\lq algebraic tensor product\rq\rq\,
consisting of all finite linear combinations of formal elements $v\otimes w$, $v\in\mathscr{M}_1^0$ and $w\in\mathscr{M}_2^0$, obtained by factorizing appropriate vector spaces (see \cite[Section~1.5]{Gigli:NonSmoothDifferentialStr}). It naturally comes with a multiplications $\cdot\,: L^0(M;\m)\times \mathscr{M}_1^0\odot\mathscr{M}_2^0 \to \mathscr{M}_1^0\odot\mathscr{M}_2^0$ defined through
\begin{align*}
f(v\otimes w):=(fv)\otimes w=v\otimes (fw)
\end{align*}
and a point-wise scalar product $\langle \cdot\,|\,\cdot\rangle_{\m}\,: (\mathscr{M}_1^0\odot\mathscr{M}_2^0)^2\to L^0(M;\m)$ given by 
\begin{align*}
\langle (v_1\otimes w_1)\,|\,(v_2\otimes w_2)\rangle:=\langle v_1,v_2\rangle_{\m}\langle w_1,w_2\rangle_{\m},
\end{align*}
both extended to $\mathscr{M}_1^0\odot\mathscr{M}_2^0$ by bi-linearity. 
Then $\langle \cdot\,|\,\cdot\rangle_{\m}$ is bilinear, $\m$-a.e.~non-negative definite, symmetric, and local in both components 
(see \cite[Lemma~3.2.19]{GPLecture}).    

The point-wise \emph{Hilbert-Schmidt norm} $|\cdot|_{{\rm HS}}: \mathscr{M}_1^0\odot\mathscr{M}_2^0\to L^0(M;\m)$ is given by 
\begin{align*}
|A|_{{\rm HS}}:=\sqrt{\langle A\,|\,A\rangle_{\m}}.
\end{align*} 
This map satisfies the $\m$-a.e. triangle inequality and is $1$-homogenous with respect to multiplication with $L^0(M;\m)$-functions (see \cite[p.~44]{Gigli:NonSmoothDifferentialStr}). 
Consequently, the map $\|\cdot\|_{\mathscr{M}_1\otimes\mathscr{M}_2}:\mathscr{M}_1^0\odot\,\mathscr{M}_2^0\to[0,+\infty]$ defined through
\begin{align*}
\|A\|_{\mathscr{M}_1\otimes\mathscr{M}_2}:=\||A|_{{\rm HS},\m}\|_{L^2(M;\m)}
\end{align*}
has all properties of a norm except that it might take the value $+\infty$. 

Now we define the \emph{tensor product} $\mathscr{M}_1\otimes\mathscr{M}_2$ the $\|\cdot\|_{\mathscr{M}_1\otimes\mathscr{M}_2}$-completion of the subspace that consists of all $A\in \mathscr{M}_1^0\odot\mathscr{M}_2^0$ such that $\|A\|_{\mathscr{M}_1\otimes\mathscr{M}_2}<\infty$.  
Denote by $A^{\top}\in \mathscr{M}^{\otimes2}$ the {\it transpose} of $A\in \mathscr{M}^{\otimes2}$ as defined in 
\cite[Section~1.5]{Gigli:NonSmoothDifferentialStr}. For instance, for bounded $v,w\in\mathscr{M}$ we have 
\begin{align}
(v\otimes w)^{\top}=w\otimes v.\label{eq:transpose}
\end{align}
}
\end{defn}

\subsection{Exterior products}\label{subsec:ExteriorProducts} 
Let $\mathscr{M}$ be a Hilbert module and $k\in\N\cup\{0\}$. 
Set $\Lambda^0\!\mathscr{M}^0:=L^0(M;\m)$ and, for $k\geq1$, let $\Lambda^k\!\mathscr{M}^0$ be the \lq\lq exterior product\rq\rq 
constructed by suitable factorizing $(\mathscr{M}^0)^{\odot k}$ (see \cite[Section~1.5]{Gigli:NonSmoothDifferentialStr}). 
The representative of $v_1\odot\cdots\odot v_k, v_1,\cdots, v_k\in\mathscr{M}^0$ in $\Lambda^k\!\mathscr{M}^0$ is written $v_1\wedge \cdots\wedge v_k$. $\Lambda^k\!\mathscr{M}^0$ naturally comes with a multiplication $\cdot :L^0(M;\m)\times\Lambda^k\!\mathscr{M}^0$ via
\begin{align*}
f(v_1\wedge \cdots\wedge v_k):=(fv_1)\wedge \cdots\wedge v_k=\cdots=v_1\wedge \cdots\wedge (fv_k)
\end{align*}
and a pointwise scalar product $\langle \cdot,\cdot\rangle_{\m}:(\Lambda^k\!\mathscr{M}^0)^2\to L^0(M;\m)$ defined by 
\begin{align}
\langle v_1\wedge\cdots \wedge v_k,w_1\wedge \cdots \wedge w_k\rangle_{\m}:={\rm det}\left[\langle v_i,w_j\rangle_{\m} \right]_{i,j\in\{1,2,\cdots,k\}}\label{eq:exteroprinnerproduct}
\end{align}
up to a factor $k!$, both extended to $\Lambda^k\!\mathscr{M}^0$ by (bi-)linearity. Then $\langle \cdot,\cdot\rangle_{\m}$ is bilinear, 
$\m$-a.e. non-negative definite, symmetric, and local in both components.

Given any $k,k'\in\N\cup\{0\}$, the map assigning to $v_1\wedge \cdots \wedge v_k\in\Lambda^k\!\mathscr{M}^0$ and 
$w_1\wedge \cdots \wedge w_{k'}\in \Lambda^{k'}\!\mathscr{M}^0$ the element 
$v_1\wedge \cdots \wedge v_k\wedge w_1\wedge \cdots \wedge w_{k'}\in \Lambda^{k+k'}\!\mathscr{M}^0$ can and will be uniquely extended by bilinearity and continuity to a bilinear map 
$\bigwedge:\Lambda^k\!\mathscr{M}^0\times\Lambda^{k'}\!\mathscr{M}^0$ termed \emph{wedge product} (see \cite[p.~47]{GPLecture}).
If $k=0$ or $k'=0$, it simply corresponds to multiplication of elements of $\Lambda^{k'}\!\mathscr{M}^0$ or 
$\Lambda^{k}\!\mathscr{M}^0$, respectively, with function in $L^0(M;\m)$ according to \eqref{eq:exteroprinnerproduct}.  

By a slight abuse of notation, define the map $|\cdot|_{\m}:\Lambda^k\!\mathscr{M}^0\to L^0(M;\m)$ by 
\begin{align*}
|\omega|_{\m}:=\sqrt{\langle \omega,\omega\rangle_{\m}}.
\end{align*}
It obeys the $\m$-a.e. triangle inequality and is  homogenous with respect to multiplication with $L^0(M;\m)$-functions (see \cite[p.~47]{GPLecture}). 

It follows that the map $\|\cdot\|_{\Lambda^k\!\mathscr{M}}:=\| |\cdot|_{\m}\|_{L^2(M;\m)}$
has all properties of a norm except that $\|w\|_{\Lambda^k\!\mathscr{M}}$ might be infinite. 
\begin{defn}\label{def:exteriorproducts}
{\rm The ($k$-fold) exterior product $\Lambda^k\!\mathscr{M}$ is defined as the completion with respect to 
$\|\cdot\|_{\Lambda^k\!\mathscr{M}}$ of the subspace consisting of all $\omega\in\Lambda^k\!\mathscr{M}^0$ such that $\|\omega\|_{\Lambda^k\!\mathscr{M}}$.
}
\end{defn}
The space $\Lambda^k\!\mathscr{M}$ naturally becomes a Hilbert module and, if $\mathscr{M}$ is separable, is separable as wel (see \cite[p.~47]{GPLecture}).

\subsection{Cotangent module}\label{subsec:CotangentModule}
Let $(\mathscr{E},D(\mathscr{E}))$ be a quasi-regular strongly local Dirichlet form on $L^2(M;\m)$. 
We define the \emph{cotangent module} $L^2(T^*\!M)$, i.e., the space of differential 
$1$-forms that are $L^2$-integrable in a certain \lq\lq universal\rq\rq\, sense. 
\begin{defn}[Pre-cotangent module]\label{df:premodule}
{\rm We define the \emph{pre-cotangent module} 
${\rm Pcm}$ by 
\begin{align*}
{\rm Pcm}:&=\Biggl\{\left.(f_i,A_i)_{i\in\mathbb{N}}\;\Biggr|\, (A_i)_{i\in\mathbb{N}}\text{ Borel partition of }X,
 \right.\\
&\hspace{3cm}
\left. (f_i)_{i\in\mathbb{N}}\subset D(\mathscr{E})_e,\; \sum_{i\in\mathbb{N}}\int_{A_i}\Gamma(f_i)\d\m<\infty \right\}
\end{align*}
Moreover, we define a relation $\sim$ on ${\rm Pcm}$ by $(f_i,A_i)_{i\in\mathbb{N}}\sim(g_j,B_j)_{j\in\mathbb{N}}$ if and only if $\int_{A_i\cap B_j}\Gamma(f_i-g_j)\d\m=0$ for every $i,j\in\mathbb{N}$. The relation, in fact forms an equivalence relation by \cite[\S2.1]{Braun:Tamed2021}. 
The equivalence class of an element $(f_i,A_i)_{i\in\mathbb{N}}\in{\rm Pcm}$ with respect to $\sim$ is denoted by $[f_i,A_i]$. The space ${\rm Pcm}/\!\!\sim$ of equivalence classes becomes a vector space via the well-defined operations 
\begin{align}
[f_i,A_i]+[g_j,B_j]:=[f_i+g_j,A_i\cap B_j],\qquad 
\lambda[f_i,A_i]:=[\lambda f_i,A_i]\label{eq:vectorSpaceequivalece}
\end{align}  
for every $[f_i,A_i],[g_j,B_j]\in {\rm Pcm}/\!\!\sim$ and $\lambda\in\R$.

Now we define the space ${\rm SF}(M;\m)\subset L^{\infty}(M;\m)$ of simple functions, i.e., each element 
$h\in {\rm SF}(M;\m)$ attains only a finite number values. For $[f_i,A_i]\in {\rm Pcm}/\!\!\sim$ 
and $h=\sum_{j=1}^{\ell}a_j\1_{B_j}\in {\rm SF}(M;\m)$ with a Borel partition $(B_j)$ of $M$, we define the product 
$h[f_i,A_i]\in {\rm Pcm}/\!\!\sim$ as 
\begin{align}
h[f_i,A_i]:=[a_jf_i,A_i\cap B_j],\label{eq:multiplicationRule1}
\end{align}
where we set $B_j:=\emptyset$ and $a_j:=0$ for every $j>{\ell}$. 
It is readily verified that this definition is well-posed and that the resulting multiplication is a bilinear map from ${\rm SF}(M;\m)\times{\rm Pcm}/\!\!\sim$ into ${\rm Pcm}/\!\!\sim$ such that for every 
$[f_i,A_i]\in {\rm Pcm}/\!\!\sim$ and every $h,k\in {\rm SF}(M;\m)$
\begin{align}
(hk)[f_i,A_i]=h(k[f_i,A_i]),\qquad 
\1[f_i,A_i]=[f_i,A_i].\label{eq:multiplicationRule2}
\end{align} 
Moreover, the map $\|\cdot\|_{L^2(T^*\!M)}:{\rm Pcm}/\!\!\sim\, \to[0,+\infty[$ given by 
\begin{align*}
\| [f_i,A_i]\|_{L^2(T^*\!M)}^2:=\sum_{i\in\mathbb{N}}\int_{A_i}\Gamma(f_i)\d\m<\infty
\end{align*}
constitutes a norm on ${\rm Pcm}/\!\!\sim$. 
}
\end{defn}

\begin{defn}[Cotangent module]\label{df:CotangentModule}
{\rm We define the Banach space $(L^2(T^*\!M),\|\cdot\|_{L^2(T^*\!M)})$ as the completion of $({\rm Pcm}/\!\!\sim, \|\cdot\|_{L^2(T^*\!M)})$. The pair $(L^2(T^*\!M),\|\cdot\|_{L^2(T^*\!M)})$  or simply 
$L^2(T^*\!M)$ is called \emph{cotangent module}, and the elements of $L^2(T^*\!M)$ are called \emph{cotangent vector fields} or (\emph{differential}) \emph{$1$-forms}. 
}
\end{defn}
The following are shown in \cite[Lemma~2.3 and Theorem~2.4]{Braun:Tamed2021}:

\begin{lem}[{\cite[Lemma~2.3]{Braun:Tamed2021}}]\label{lem:ModulePropoerty}
The map from ${\rm SF}(M)\times {\rm Pcm}/\!\!\sim$ into ${\rm Pcm}/\!\!\sim$ defined in \eqref{eq:multiplicationRule1} 
extends continuously and uniquely to a bilinear map from $L^{\infty}(M;\m)\times L^2(T^*\!M)$ into $L^2(T^*\!M)$ satisfying, for every $f,g\in L^{\infty}(M;\m)$ and every $\omega\in L^2(T^*\!M)$,
\begin{align*}
(fg)\omega=f(g\omega),\quad \1_M \omega=\omega,\quad \|f\omega\|_{L^2(T^*\!M)}\leq \|f\|_{L^{\infty}(M;\m)}\cdot\|\omega\|_{L^2(T^*\!M)}.
\end{align*}
\end{lem}
\begin{thm}[{Module property, \cite[Theorem~2.4]{Braun:Tamed2021}}]\label{thm:ModuleProperty}
The cotangent module $L^2(T^*\!M)$ is an $L^2$-normed $L^{\infty}$-module over $M$ with respect to $\m$ whose point-wise norm $|\cdot|_{\m}$ satisfies, for every $[f_i,A_i]\in{\rm Pcm}/\!\!\sim$,
\begin{align}
|[f_i,A_i]|_{\m}=\sum_{i\in\mathbb{N}}\1_{A_i}\Gamma(f_i)^{\frac12}\quad\m\text{-a.e.}\label{eq:ModuleProperty}
\end{align}
In particular, $L^2(T^*\!M)$ is a Hilbert module with respect to $\m$. 
\end{thm}
\begin{defn}[$L^2$-differential]\label{df:differetial}
{\rm The $L^2$-differential $\d f$ of any function $f\in D(\mathscr{E})_e$ is defined by 
\begin{align*}
{\d f}:=[f,X]\in L^2(T^*\!M),
\end{align*}
where $[f,X]\in {\rm Pcm}/\!\!\sim\;\subset L^2(T^*\!M)$ 
is the representative of the sequence $(f_i,A_i)_{i\in\mathbb{N}}$ given by 
$f_i:=f$, $A_1:=X$, $f_i:=0$ and $A_i:=\emptyset$ for every $i\geq2$. 

\bigskip

As usual, we call a $1$-form $\omega\in  L^2(T^*\!M)$ \emph{exact} if, for some $f\in D(\mathscr{E})_e$,
\begin{align*}
\omega=\d f.
\end{align*}
The $L^2$-differential $\d$ is a linear operator on $D(\mathscr{E})_e$. By \eqref{eq:ModuleProperty}, the $L^{\infty}$-module structure induced by $\m$ according to Theorem~\ref{thm:ModuleProperty},
\begin{align*}
|\d f|_{\m}=\Gamma(f)^{\frac12}\quad\m\text{-a.e.}
\end{align*}
holds for every $f\in D(\mathscr{E})_e$. 
}
\end{defn}
\subsection{Tangent module}\label{subsec:TangentModule}
In this section, let $(M,{\sf d},\m)$ be a metric measure space and assume the 
infinitesimally Hilbertian condition for it. In particular, we are given a strongly local Dirichlet form 
$(\mathscr{E},D(\mathscr{E}))$ on $L^2(M;\m)$. 
We define the notion of \emph{tangent module}. 
\begin{defn}[Tangent module]\label{df:TangentModule}
{\rm The tangent module $(L^2(TM),\|\cdot\|_{L^2(TM)})$ or simply $L^2(TM)$ is 
\begin{align*}
L^2(TM):=L^2(T^*\!M)^*
\end{align*}
and it is endowed with the norm $\|\cdot\|_{L^2(TM)}$ induced by \eqref{eq:DualPointNorm}. 
The elements of $L^2(TM)$ will be called \emph{vector fields}.  
}
\end{defn}
 As in Subsection~\ref{subsec:normedModule}, the point-wise pairing between $\omega\in L^2(T^*\!M)$ and $X\in L^2(TM)$ is denoted by $\omega(X)\in L^1(M;\m)$, and, by a slight abuse of notation, 
 $|X|\in L^2(M;\m)$ denotes the point-wise norm of $M$. By \cite[Lemma~2.7 and Proposition~1.24]{Braun:Tamed2021}, 
 $L^2(TM)$ is a separable Hilbert module. Furthermore, in terms of the point-wise scalar product $\langle \cdot,\cdot\rangle $ on $L^2(T^*\!M)$ and $L^2(TM)$, respectively, \cite[Proposition~1.24]{Braun:Tamed2021} allows us to define the \emph{(Riesz) musical isomorphisms} $\sharp: L^2(T^*\!M)\to L^2(TM)$ 
 and $\flat:=\sharp^{-1}$ defined by 
 \begin{align}
 \langle \omega^{\sharp},X\rangle :=\omega(X)=:\langle X^{\flat},\omega\rangle \quad \m\text{-a.e.}\label{eq:MusicalMap}
 \end{align}
\begin{defn}[$L^2$-gradient]\label{df:Gradient}
{\rm The \emph{$L^2$-gradient} $\nabla f$ of a function $f\in D(\mathscr{E})_e$ is defined by 
\begin{align*}
\nabla f:=(\d f)^{\sharp}.
\end{align*}
Observe from \eqref{eq:MusicalMap} that $f\in D(\mathscr{E})_e$, is characterized as the unique element 
$X\in L^2(TM)$ which satisfies 
\begin{align*}
\d f(X)=|\d f|^2=|X|^2\quad\m\text{-a.e.}
\end{align*}
}
\end{defn} 

\subsection{Divergences}\label{subsec:divergence}

\begin{defn}[$L^2$-divergence]\label{df:L2Divergence}
{\rm We define the space $D({\rm div})\subset L^2(TM)$ by 
\begin{align*}
D({\rm div}):&=\Bigl\{X\in  L^2(TM)\,\Bigl|\, \text{ there exists a function }f\in L^2(M;\m)\\ 
&\hspace{2cm}\text{ such  that for every }h\in D(\mathscr{E}),\quad -\int_M  h\,f\,\d\m=\int_M \d h(X)\d\m\Bigr\}.
\end{align*} 
Such $f\in L^2(M;\m)$ is unique, and it is called the \emph{{\rm(}$L^2$-{\rm)}divergence of} $M$ and denoted by ${\rm div}\,X$.
}
\end{defn} 
The uniqueness of ${\rm div}\,X$ comes from the denseness of $D(\mathscr{E})$ in $L^2(M;\m)$.     
Note that the map ${\rm div}:D({\rm div})\to L^2(M;\m)$ is linear, hence $D({\rm div})$ is a vector space. By definition of $D(\Delta)$, $\nabla D(\Delta)\subset D({\rm div})$ and 
\begin{align}
{\rm div}\nabla f=\Delta f\quad \m\text{-a.e.}\label{eq:Divergence}
\end{align}
for every $f\in D(\Delta)$. Moreover, using the Leibniz rule in \cite[Theorem~4.4]{GPLecture},
one can verify that for every $X\in D({\rm div})$ and $f\in D(\mathscr{E})_e\cap L^{\infty}(M;\m)$ with 
$|\d f|\in L^{\infty}(M;\m)$, we have $fX\in D({\rm div})$ and 
\begin{align}
{\rm div}(fX)=f\,{\rm div}(X)+\d f(X)\quad\m\text{-a.e.}\label{eq:DivergenceFormula}
\end{align}

\subsection{The Lebesgue spaces $L^p(T^*\!M)$ and $L^p(TM)$}\label{subsec:LebesgueSpace}
Let $L^0(T^*\!M)$ and $L^0(TM)$ be the $L^0$-modules as in Definition~\ref{df:L0Module} associated to $L^2(T^*\!M)$ and $L^2(TM)$, i.e.,
\begin{align*}
L^0(T^*\!M):=L^2(T^*\!M)^0,\quad  L^0(TM):=L^2(TM)^0.
\end{align*}
The characterization of Cauchy sequences in these spaces grants that 
the point-wise norms $|\cdot|_{\m}:L^2(T^*\!M)\to L^2(M;\m)$ and $|\cdot|_{\m}:L^2(TM)\to L^2(M;\m)$ as well as the (Riesz) musical isomorphisms $\flat:L^2(TM)\to L^2(T^*\!M)$ and $\sharp:L^2(T^*\!M)\to L^2(TM)$ uniquely extend to (non-relabeled) continuous map $\flat:L^0(TM)\to L^0(T^*\!M)$ and $\sharp:L^0(T^*\!M)\to L^0(TM)$. 

For $p\in[1,+\infty]$, let $L^p(T^*\!M)$ and $L^p(TM)$ be the Banach spaces consisting of all $\omega\in 
L^0(T^*\!M)$ and $X\in L^0(TM)$ such that $|\omega|\in L^p(M;\m)$ and $|X|\in L^p(M;\m)$, respectively, endowed with the norms
\begin{align*}
\|\omega\|_{L^p(T^*\!M)}:=\||\omega|_{\m}\|_{L^p(M;\m)},\qquad
\|X\|_{L^p(TM)}:=\||X|_{\m}\|_{L^p(M;\m)}.
\end{align*}
Since by \cite[Lemma~2.7]{Braun:Tamed2021}, $L^2(T^*\!M)$ is separable, and so is $L^2(TM)$ by \cite[Proposition~1.24]{Braun:Tamed2021}. Then one easily derives that if $p\ne+\infty$, the spaces 
$L^p(T^*\!M)$ and $L^p(TM)$ are separable as well. Since $L^2(T^*\!M)$ and $L^2(TM)$ are reflexive as Hilbert spaces, by the discussion in Definition~\ref{df:DualModule} it follows that $L^p(T^*\!M)$ and $L^p(TM)$ are reflexive for every $p\in]1,+\infty[$. For $q\in[1,+\infty[$ such that $1/p+1/q=1$ in the sense of $L^{\infty}$-modules we have the duality
\begin{align*}
L^p(T^*\!M)^*=L^q(TM).
\end{align*}
The point-wise norms $|\cdot|_{\m}:L^p(T^*\!M)\to L^p(M;\m)$ and $|\cdot|_{\m}:L^p(TM)\to L^p(M;\m)$ as well as the (Riesz) musical isomorphisms $\flat:L^p(TM)\to L^p(T^*\!M)$ and $\sharp:L^p(T^*\!M)\to L^p(TM)$ are defined by 
\begin{align}
{}_{L^p(TM)}\langle \omega^{\sharp},X\rangle_{L^q(TM)}:=\omega(X)=:{}_{L^q(T^*\!M)}\langle X^{\flat},\omega\rangle_{L^p(T^*\!M)}\quad \m\text{-a.e.}\label{eq:couplingP}
\end{align}
for $X\in L^q(TM)$ and $\omega\in L^p(T^*\!M)$,  
and we write \eqref{eq:couplingP} by 
\begin{align}
\langle \omega^{\sharp},X\rangle :=\omega(X)=:\langle X^{\flat},\omega\rangle\quad \m\text{-a.e.}\label{eq:couplingP*}
\end{align}
for simplicity.

\subsection{Test and regular objects}\label{subsec:TestandRegularObjects}
\begin{defn}[${\rm Test}(TM)$ and ${\rm Reg}(TM)$]\label{def:TestVecFields}
{\rm We define the subclass of $L^2(TM)$ consisting of \emph{test vector fields} or \emph{regular vector fields}, respectively: 
\begin{align*}
{\rm Test}(TM)&=\left\{\left.\sum_{i=1}^ng_i\nabla f_i\,\right|\, n\in\N,f_i,g_i\in {\rm Test}(M) \right\},\\
{\rm Reg}(TM)&=\left\{\left.\sum_{i=1}^ng_i\nabla f_i\,\right|\, n\in\N,f_i\in {\rm Test}(M),g_i\in {\rm Test}(M)\cup\R\1_M  \right\}.
\end{align*}

}
\end{defn}

\begin{defn}[${\rm Test}(T^*\!M)$ and ${\rm Reg}(T^*\!M)$]\label{def:Test1Forms}
{\rm We define the subclass of $L^2(T^*\!M)$ consisting of \emph{test $1$-forms} or \emph{regular $1$-forms}, respectively:
\begin{align*}
{\rm Test}(T^*\!M)&=\left\{\left.\sum_{i=1}^ng_i\d f_i\,\right|\, n\in\N,f_i,g_i\in {\rm Test}(M) \right\},\\
{\rm Reg}(T^*\!M)&=\left\{\left.\sum_{i=1}^ng_i\d f_i\,\right|\, n\in\N,f_i\in {\rm Test}(M),g_i\in {\rm Test}(M)\cup\R\1_M  \right\}.
\end{align*}
${\rm Test}(TM)$ (resp.~${\rm Test}(T^*\!M)$) is a dense subspace of $L^p(TM)$ (resp.~$L^p(T^*\!M)$) for $p\in[1,+\infty[$, hence 
${\rm Reg}(TM)\cap L^p(TM)$ (resp.~${\rm Reg}(T^*\!M)\cap L^p(T^*\!M)$) is dense in $L^p(TM)$ (resp.~$L^p(T^*\!M)$) (see Section~\ref{sec:cor:HessShcraderUhlenbrock}). 
From this, we have that $L^2(TM)\cap L^p(TM)$ (resp.~$L^2(T^*\!M)\cap L^p(T^*\!M)$) is dense in $L^p(TM)$ 
(resp.~$L^p(T^*\!M)$) for $p\in[1,+\infty[$.
}
\end{defn}
\subsection{Lebesgue spaces on tensor products}\label{subsec:LebesgueSpaceTensorProdulct}
Denote the two-fold tensor products of $L^2(T^*\!M)$ and $L^2(TM)$, respectively, in the sense of Definition~\ref{df:TensorProducts} by 
\begin{align*}
L^2((T^*)^{\otimes2}X):=L^2(T^*\!M)\otimes L^2(T^*\!M),\quad 
L^2((T)^{\otimes2}X):=L^2(TM)\otimes L^2(TM).
\end{align*}
By the discussion from Subsection~\ref{subsec:TensorProducts},
Theorem~\ref{thm:ModuleProperty} and 
\cite[Proposition~1.24]{Braun:Tamed2021}, both are separable modules. They are point-wise isometrically module isomorphic: the respective pairing is initially defined by 
\begin{align*}
(\omega_1\otimes\omega_2)(X_1\otimes X_2):=\omega_1(X_1)\omega_2(X_2)\quad\m\text{-a.e.}
\end{align*}
for $\omega_1,\omega_2\in L^2(T^*\!M)\cap L^{\infty}(T^*\!M)$ and $X_1,X_2\in L^2(TM)\cap L^{\infty}(TM)$, 
and is extended by linearity and continuity to $L^2((T^*)^{\otimes2}X)$ and $L^2((T)^{\otimes2}X)$, respectively. By a slight abuse of notation, this pairing, with \cite[Proposition~1.24]{Braun:Tamed2021}, induces the (Riesz) musical isomorphisms $\flat:L^2((T)^{\otimes2}X)\to L^2((T^*)^{\otimes2}X)$ and $\sharp:=\flat^{-1}$ given by 
\begin{align}
\langle A^{\sharp}\,|\, T\rangle_{\m}:=A(T)=: \langle A\,|\,T^{\flat}\rangle_{\m}\quad\m\text{-a.e.}\label{eq:musical2}
\end{align}
and  
write $|A|_{\rm HS}:=\sqrt{\langle A\,|\,A\rangle_{\m}}$ and $|T|_{\rm HS}:=\sqrt{\langle T\,|\,T\rangle_{\m}}$ for $A\in L^2((T^*)^{\otimes2}X)$ and $T\in L^2((T)^{\otimes2}X)$. 

We let $L^p((T^*)^{\otimes2}X)$ and $L^p(T^{\otimes2}X)$, $p\in\{0\}\cup[1,+\infty]$, can be defined similarly as in 
Subsection~\ref{subsec:LebesgueSpace}. For $p\in[1,+\infty]$, these spaces naturally become Banach space which, if $p<\infty$, are separable. The following coincidences also hold 
\begin{align*}
L^p((T^*)^{\otimes2}X)=L^p(T^*\!M)\otimes L^p(T^*\!M),\quad 
L^q((T)^{\otimes2}X):=L^q(TM)\otimes L^q(TM)
\end{align*}
and  the (Riesz) musical isomorphisms $\flat:L^q((T)^{\otimes2}X)\to L^p((T^*)^{\otimes2}X)$ and $\sharp:=\flat^{-1}$ can be defined by 
\eqref{eq:musical2} 
for $A\in L^p((T^*)^{\otimes2}X)$ and $T\in L^q((T)^{\otimes2}X)$. 

\bigskip

We define the $L^p$-dense sets, $p\in[1,+\infty]$, intended strongly if $p<\infty$ and weakly* if $p=\infty$, reminiscent of 
Subsection~\ref{subsec:TensorProducts}. 
\begin{align*}
{\rm Test}((T^*)^{\otimes2}X):={\rm Test}(T^*\!M)^{\odot2},\\
{\rm Test}(T^{\otimes2}X):={\rm Test}(TM)^{\odot2},\\
{\rm Reg}((T^*)^{\otimes2}X):={\rm Reg}(T^*\!M)^{\odot2},\\
{\rm Reg}(T^{\otimes2}X):={\rm Reg}(TM)^{\odot2}.
\end{align*}
But the denseness in $L^p((T^*)^{\otimes2}X)$ (resp.~$L^p(TM)^{\otimes2}$) 
for ${\rm Reg}((T^*)^{\otimes2}X)$ (resp.~${\rm Reg}(T^{\otimes2}X)$) should be understood as 
for ${\rm Reg}((T^*)^{\otimes2}X)\cap L^p((T^*)^{\otimes2}X)$ (resp.~${\rm Reg}(T^{\otimes2}X)\cap L^p(T^{\otimes2}X)$).

\subsection{Lebesgue spaces on exterior products}\label{subsec:LebesgueSpExtPro}
Given any $k\in\N\cup\{0\}$, we set 
\begin{align*}
L^2(\Lambda^kT^*\!M):&=\Lambda^k L^2(T^*\!M),\\
L^2(\Lambda^kTM):&=\Lambda^k L^2(TM),
\end{align*}
where the exterior products are defined in Subsection~\ref{subsec:ExteriorProducts}. For $k\in\{0,1\}$, we see 
\begin{align*}
L^2(\Lambda^1T^*\!M)&=L^2(T^*\!M),\\
L^2(\Lambda^1TM)&=L^2(TM),\\
L^2(\Lambda^0T^*\!M)&=L^2(\Lambda^0TM)=L^2(M;\m).
\end{align*}
By Subsection~\ref{subsec:ExteriorProducts}, these are naturally Hilbert modules. As in Subsection~\ref{subsec:LebesgueSpaceTensorProdulct}, 
$L^2(\Lambda^kT^*\!M)$ and $L^2(\Lambda^kTM)$ are pointwise isometrically module isomorphic. 
For brevity, the induced pointwise pairing between $\omega\in L^2(\Lambda^kT^*\!M)$ and $X_1\land \cdots\land X_k\in 
L^2(\Lambda^k TM)$ with $X_1,\cdots, X_k\in L^2(TM)\cap L^{\infty}(TM)$, is written by 
\begin{align*}
\omega(X_1,\cdots, X_k):=\omega(X_1\land\cdots \land X_k). 
\end{align*}
We let $L^p(\Lambda^kT^*\!M)$ and $L^p(\Lambda^k TM)$, $p\in\{0\}\cup[1,+\infty]$, be as in Subsection~\ref{subsec:LebesgueSpace}. 
For $p\in[1,+\infty]$, these spaces are Banach and, if $p<\infty$, additionally separable.  
We define the formal $k$-th exterior products, $k\in\N\cup\{0\}$, of the classes from Subsection~\ref{subsec:TestandRegularObjects} as follows: 
\begin{align*}
{\rm Test}(\Lambda^kT^*\!M):&=\left\{\left. \sum_{i=1}^n f_i^0\d f_i^1\land\cdots\land \d f_i^k  \;\right|\; 
n\in\N, f_i^j\in{\rm Test}(M)\text{ for }0\leq j\leq k \right\},\\
{\rm Test}(\Lambda^kTM):&=\left\{\left. \sum_{i=1}^n f_i^0\nabla f_i^1\land\cdots\land \nabla f_i^k  \;\right|\; 
n\in\N, f_i^j\in{\rm Test}(M)\text{ for }0\leq j\leq k \right\},\\
{\rm Reg}(\Lambda^kT^*\!M):&=\left\{\left. \sum_{i=1}^n f_i^0\d f_i^1\land\cdots\land \d f_i^k  \;\right|\; 
n\in\N, f_i^j\in{\rm Test}(M)\text{ for }1\leq j\leq k,\right. \\
&\hspace{9cm} f_i^0\in{\rm Test}(M)\cup\R\1_M  \Biggr\},\\
{\rm Reg}(\Lambda^kTM):&=\left\{\left. \sum_{i=1}^n f_i^0\d f_i^1\land\cdots\land \d f_i^k  \;\right|\; 
n\in\N, f_i^j\in{\rm Test}(M)\text{ for }1\leq j\leq k, \right.\\ 
&\hspace{9cm}
f_i^0\in{\rm Test}(M)\cup\R\1_M  \Biggr\}. 
\end{align*}
We employ the evident interpretations for $k=1$, while the respective spaces for $k=0$ are identified with those spaces to which their generic elements's zeroth order terms belong to. These classes are dense in their respective $L^p$-spaces, 
$p\in[1,+\infty]$, intended strongly if $p<\infty$ and weakly* if $p=\infty$.  
But the denseness in $L^p(\Lambda^kT^*\!M)$ (resp.~$L^p(\Lambda^kTM)$) for ${\rm Reg}(\Lambda^kT^*\!M)$ (resp.~${\rm Reg}(\Lambda^kTM)$) should be modified as noted before.

\section{Covariant derivative}

\subsection{The Sobolev spaces $W^{1,2}(TM)$ and $H^{1,2}(TM)$}\label{subsec:W12SobolevTX}
We construct the tangent module $L^2(TM)$ in Subsection~\ref{subsec:TangentModule}. Now we define the $(1,2)$-Sobolev space 
$W^{1,2}(TM)$ based on the notion of $L^2$-Hessian. 
\begin{defn}[$(1,2)$-Sobolev space $W^{1,2}(TM)$]\label{def:SobolevSpaceW12TX}
{\rm The space $W^{1,2}(TM)$ is defined to consist of all $X\in L^2(TM)$ for which there exists a $T\in L^2(T^{\otimes2}X)$
such that for every $g_1,g_2,h\in {\rm Test}(M)$,
\begin{align*}
\int_M h\langle T\,|\,&\nabla g_1\otimes\nabla g_2\rangle \d\m\\
&=-\int_M \langle X,\nabla g_2\rangle {\rm div}(h\nabla g_1)\d\m-\int_M \,h\,{\rm Hess}\,g_2(X,\nabla g_1)\d\m.
\end{align*}
Here ${\rm Hess}\,g_2\in L^2((T^*)^{\otimes2}X)$ is the \emph{Hessian} defined for $g_2\in {\rm Test}(M)$ (see \cite[Definition~5.2]{Braun:Tamed2021}).  
In case of existence, the element $T$ is unique, denoted by $\nabla X$ and called the \emph{covariant derivative} of $M$. 
}
\end{defn} 

Arguing as in \cite[after Definition~5.2]{Braun:Tamed2021}, the uniqueness statement in Definition~\ref{subsec:W12SobolevTX} is derived. 
In particular, $W^{1,2}(TM)$ constitutes a vector space and the covariant derivative $\nabla$ is a linear operator on it. 
The space $W^{1,2}(TM)$ is endowed with the norm $\|\cdot\|_{W^{1,2}(M)}$ given by
\begin{align*}
\|X\|_{W^{1,2}(TM)}^2:=\|X\|_{L^2(TM)}^2+\|\nabla X\|_{L^2(T^{\otimes2}X)}^2.
\end{align*}
We also define the \emph{covariant functional} $\mathscr{E}_{\rm cov}:L^2(TM)\to [0,+\infty[$ by
\begin{align}
\mathscr{E}_{\rm cov}(X):=\left\{\begin{array}{cc}\displaystyle{\int_M |\nabla X|_{\rm HS}^2\d\m} & \text{ if }X\in W^{1,2}(TM), \\ \infty & \text{ otherwise. }\end{array}\right.\label{eq:covariantfunctional}
\end{align}
It is proved in \cite[Theorem~6.3]{Braun:Tamed2021} that 
$(W^{1,2}(TM),\|\cdot\|_{W^{1,2}(TM)})$ is a separable Hilbert space, $\nabla$ is a closed operator, ${\rm Reg}(TM)\subset W^{1,2}(TM)$, $\|\cdot\|_{L^2(TM)}$-denseness of $W^{1,2}(TM)$ in $L^2(TM)$, and the lower semi continuity of $\mathscr{E}_{\rm cov}: L^2(TM)\to[0,+\infty[$. 
 
\begin{defn}[$(1,2)$-Sobolev space $H^{1,2}(TM)$]\label{def:SobolevSpaceH12TX}
{\rm We define the space $H^{1,2}(TM)\subset W^{1,2}(TM)$ as the $\|\cdot\|_{W^{12,}(TM)}$-closure of ${\rm Reg}(TM)$:
\begin{align*}
H^{1,2}(TM):=\overline{{\rm Reg}(TM)}^{\|\cdot\|_{W^{1,2}(TM)}}.
\end{align*}
$H^{1,2}(TM)$ is in general a strict subset of $W^{1,2}(TM)$. 
}
\end{defn} 
 
The following lemma is a version of what is known as \emph{Kato's inequality} (for the Bochner Laplacian) 
in the smooth case (see \cite[Chapter~2]{HSU}, \cite[Lemma~3.5]{GP}).  
\begin{lem}[{Kato's inequality, \cite[Lemma~6.12]{Braun:Tamed2021}}]\label{lem:KatoIneq}
For every $X\in H^{1,2}(TM)$, $|X|\in D(\mathscr{E})$ and
\begin{align*}
|\nabla |X||\leq|\nabla X|_{\rm HS}\quad\m\text{-a.e.}
\end{align*}
In particular, if $X\in H^{1,2}(TM)\cap L^{\infty}(M;\m)$, then $|X|^2\in D(\mathscr{E})$.
\end{lem}

\subsection{Bochner Laplacian}\label{subsec:BochnerLaplacian} 
We have a preliminary choice to make, i.e. either to define the Bochner Laplacian $\square^B$ on $W^{1,2}(TM)$ or on the \emph{strictly small space} $H^{1,2}(TM)$. We choose the latter one since the calculation rules is more powerful. 

\begin{defn}[{Bochner Laplacian}]\label{def:BochnerLaplacian}
{\rm We define $D(\square^B)$ to consist of all $X\in H^{1,2}(TM)$ for which there exists $Z\in L^2(TM)$ such that 
for every $Y\in H^{1,2}(TM)$, 
\begin{align*}
\int_M \langle Y,Z\rangle \d\m=-\int_M \langle \nabla Y\,|\,\nabla X\rangle \d\m. 
\end{align*}
In case of existence, $Z$ is uniquely determined, denoted by $\square^B X$ and called the \emph{Bochner Laplacian} (or \emph{connection Laplacian} or \emph{horizontal Laplacian}) of $M$. 
}
\end{defn} 
observe that $D(\square^B)$ is a vector space, and that $\square^B:D(\square^B)\to L^2(TM)$ is a linear operator. Both are easy to see from the linearity of the covariant derivative. 

We modify the functional from \eqref{eq:covariantfunctional} with the domain $W^{1,2}(TM)$ by introducing the \lq\lq augmented\rq\rq\, covariant energy functional $\widetilde{\mathscr{E}}_{\rm cov}:L^2(TM)\to [0,+\infty]$ with 
\begin{align}
\widetilde{\mathscr{E}}_{\rm cov}(X):=\left\{\begin{array}{cc}\displaystyle{\int_M|\nabla X|_{\rm HS}^2\d\m} & \text{ if }X\in H^{1,2}(TM), \\ \infty & \text{ otherwise. }\end{array}\right.\label{eq:covariantfunctional*}
\end{align}
Clearly, its (non-relabeled) polarization $\widetilde{\mathscr{E}}_{\rm cov}:H^{1,2}(TM)^2\to\R$ is a closed, symmetric form, and $\square^B$ is the non-positive, self-adjoint generator uniquely associated to it according to \cite[Theorem~3.3]{FOT}.  
We can write $\mathscr{E}^B$ instead of $\widetilde{\mathscr{E}}_{\rm cov}$. 
Let $(P_t^{\rm B})_{t\geq0}$ be the heat semigroup of bounded linear and self-adjoint operator on $L^2(TM)$ formally written by 
\begin{align*}
\text{\lq\lq $P_t^{\rm B}:=e^{t\,\square^{\rm B}}$\rq\rq}.
\end{align*}
For $\alpha>0$ and $X\in L^2(TM)$, 
we set $R_{\alpha}^{\rm B}X:=\int_0^{\infty}e^{-\alpha t}P_t^{\rm B}X \d t$.
\begin{lem}[{\cite[Corollary~6.21, Theorem~6.26]{Braun:Tamed2021}}]\label{lem:bochnerHF}
We have the following: 
\begin{enumerate}
\item[\rm(1)] $0\leq\inf\sigma(-\Delta)\leq\inf\sigma(-\square^{\rm B})$. 
\item[\rm(2)] For every $X\in L^2(TM)$ and every $t\geq0$,
\begin{align}
|P_t^{\rm B}X|\leq P_t|X|\quad\m\text{-a.e.}\label{eq:BochnerContraction}
\end{align}
\end{enumerate}
\end{lem}

\begin{cor}[{\cite[Corollary~6.29]{Braun:Tamed2021}}]\label{cor:strongcontiBochnerSemigroup}
Suppose $p\in[1,+\infty[$. Then the heat flow $(P_t^{\rm B})_{t\geq0}$ can be extended to a contractive 
semigroup on $L^p(TM)$, which is strongly continuous on $L^p(TM)$ under $p\in[1,+\infty[$ and 
weakly* continuous on $L^{\infty}(TM)$.
\end{cor}

\section{Exterior derivative}\label{sec:ExteriorDerivative}
Throughout this section, we fix $k\in\N\cup\{0\}$.  
\subsection{The Sobolev space $D(\d^k)$ and $D(\d_*^k)$}\label{subsec:ExteriorDerivSobolev}

Given $\omega\in L^0(\Lambda^k T^*\!M)$ and $X_0,\cdots, X_k,Y\in L^0(TM)$, we shall use the standard abbreviations: for $1\leq i<j\leq k$ 
\begin{align*}
\omega(\widehat{X}_i):&=\omega(X_0,\cdots,\widehat{X}_i,\cdots, X_k),\\
:&=\omega(X_0\land\cdots\land X_{i-1}\land X_{i+1}\land\cdots\land X_k),\\
\omega(Y,\widehat{X}_i,\widehat{Y}_j):&=\omega(Y,X_0,\cdots,\widehat{X}_i,\cdots,\widehat{X}_j,\cdots, X_k),\\
:&=\omega(Y\land X_0\land\cdots\land X_{i-1}\land X_{i+1}\land\cdots\land Y_{j-1}\land Y_{j+1}\land\cdots\land X_k).
\end{align*}

\begin{defn}[Sobolev space $D(\d^k)$]\label{def:d}
{\rm We define $D(\d^k)$ to consist of all $\omega\in L^2(\Lambda^kT^*\!M)$ for which there exists $\eta\in L^2(\Lambda^{k+1}T^*\!M)$ such that for every $X_0,\cdots, X_k\in {\rm Test}(M)$, 
\begin{align*}
\int_M \eta(X_0,\cdots, X_k)\d\m&=\int_M\sum_{i=0}^k(-1)^{i+1}\omega(\widehat{X}_i){\rm div}\,X_i\d\m\\
&\hspace{2cm}+\int_M\sum_{i=0}^k\sum_{j=i+1}^k(-1)^{i+j}\omega([X_i,X_j],\widehat{X}_i,\widehat{X}_j)\d\m.
\end{align*}
In case of existence, the element $\eta$ is unique, denoted by $\d\omega$ and called the \emph{exterior derivative} 
(or \emph{exterior differential}) of $\omega$. 
}
\end{defn}

The uniqueness follows by density of ${\rm Test}(\Lambda^{k+1}T^*\!M)$ in $L^2(\Lambda^{k+1}T^*\!M)$ as discussed in 
Section~\ref{subsec:LebesgueSpace}. It is then clear that  $D(\d^k)$ is a real vector space and that $\d$ is 
a linear operator on it.

We always endow  $D(\d^k)$ with the norm $\|\cdot\|_{D(\d^k)}$ given by 
\begin{align*}
\|\omega\|_{D(\d^k)}^2:=\|\omega\|_{L^2(\Lambda^k T^*\!M)}^2+\|\d\omega\|_{L^2(\Lambda^k T^*\!M)}^2.
\end{align*}

We introduce the functional $\mathscr{E}_{\d}: L^2(\Lambda^kT^*\!M)\to[0,+\infty]$ with
\begin{align*}
\mathscr{E}_{\d}(\omega):=\left\{\begin{array}{cc}\displaystyle{\int_M |\d\omega|^2\d\m} & \text{ if }\omega\in D(\d^k), \\ \infty & \text{ otherwise.} \end{array}\right.
\end{align*} 
We do not make explicit the dependency of $\mathscr{E}_{\d}$ on the degree $k$. It will always be clear 
from the context which one is intended. It is proved in \cite[Theorem~7.5]{Braun:Tamed2021} that 
$(D(\d^k),\|\cdot\|_{D(\d^k)})$  is a separable Hilbert space, the exterior differential $\d$ is a closed operator, 
${\rm Reg}(\Lambda^kT^*\!M)\subset D(\d^k)$, $D(\d^k)$ is dense in $L^2 (\Lambda^kT^*\!M)$, and the functional 
$\mathscr{E}_{\d}:L^2(\Lambda^kT^*\!M)\to[0,+\infty]$ is lower semi continuous. 

\begin{defn}[The space $D_{\rm reg}(\d^k)$]\label{def:d*}
{\rm We define the space $D_{\rm reg}(\d^k)\subset D(\d^k)$ by the closure of ${\rm Reg}(\Lambda^kT^*\!M)$ with respect to the norm $\|\cdot\|_{D(\d^k)}$:
\begin{align*}
D_{\rm reg}(\d^k):=\overline{{\rm Reg}(\Lambda^kT^*\!M)}^{\|\cdot\|_{D(\d^k)}}.
\end{align*}
It is proved in \cite[Theorem~7.5]{Braun:Tamed2021} that for every $\omega\in D_{\rm reg}(\d^k)$, we have 
$\d\omega\in D_{\rm reg}(\d^{k+1})$ with $\d(\d\omega)=0$. 
}
\end{defn}

\begin{defn}[The space $D(\d_*^k)$]\label{def:adjointextrioderivative}
{\rm Given any $k\in\N$, the space  $D(\d_*^k)$ is defined to consist of all $\omega\in L^2(\Lambda^kT^*\!M)$ for which there exists $\rho\in L^2(\Lambda^{k-1}T^*\!M)$ such that for every $\eta\in {\rm Test}(\Lambda^{k-1}T^*\!M)$, we have 
\begin{align*}
\int_M \langle \rho,\eta\rangle \,\d\m=\int_M \langle \omega,\d\eta\rangle \,\d\m. 
\end{align*}
If it exists, $\rho$ is unique, denoted by $\d_*\omega$ and called the \emph{codifferential} of $\omega$. We simply define 
$D(\d_*^0):=L^0(M;\m)$ and $\d_*:=0$ on this space. 
}
\end{defn}
By the density of ${\rm Test}(\Lambda^{k-1}T^*\!M)$ in $L^2(\Lambda^{k-1}T^*\!M)$, the uniqueness statement is indeed true. 
Furthermore, $\d_*$ is a closed operator, i.e., the image of the assignment ${\rm Id}\times\d_*:D(\d_*^k)\to L^2(\Lambda^kT^*\!M)\times L^2(\Lambda^{k-1}T^*\!M)$ is closed in $L^2(\Lambda^kT^*\!M)\times L^2(\Lambda^{k-1}T^*\!M)$. 

By way of \cite[Theorem~7.5 and Lemma~7.15]{Braun:Tamed2021}, we have the following: 
For $f\in D(\mathscr{E})_e\cap L^{\infty}(M;\m)$ and $\omega\in H^{1,2}(T^*\!M)$ satisfying $\d f\in L^{\infty}(M;\m)$ or $\omega\in L^{\infty}(T^*\!M)$, 
$f\omega\in H^{1,2}(T^*\!M)$ with
\begin{align}
\d(f\omega)&=f\d\omega+\d f\land \omega,\\
\d_*(f\omega)&=f\d_*\omega-\langle \d f,\omega\rangle.
\end{align}

\subsection{The Sobolev spaces $W^{1,p}(\Lambda^kT^*\!M)$ and $H^{1,p}(\Lambda^{k+1}T^*\!M)$ 
}\label{subsec:SobolevExteriorDual}

\begin{defn}[The space $W^{1,2}(\Lambda^kT^*\!M)$]
{\rm We define the space $W^{1,2}(\Lambda^kT^*\!M)$ by 
\begin{align*}
W^{1,2}(\Lambda^kT^*\!M):=D(\d^k)\cap D(\d_*^k).
\end{align*}
By \cite[Theorem~7.5 and Lemma~7.15]{Braun:Tamed2021}, we already know that $W^{1,2}(\Lambda^kT^*\!M)$ is a dense subspace of 
$L^2(\Lambda^kT^*\!M)$. 

We endow $W^{1,2}(\Lambda^kT^*\!M)$ with the norm $\|\cdot\|_{W^{1,2}(\Lambda^kT^*\!M)}$ given by 
\begin{align*}
\|\omega\|_{W^{1,2}(\Lambda^kT^*\!M)}^2:=\|\omega\|_{L^2(\Lambda^kT^*\!M)}^2+\|\d \omega\|_{L^2(\Lambda^{k+1}T^*\!M)}^2+
\|\d_*\omega\|_{L^2(\Lambda^{k-1}T^*\!M)}^2
\end{align*}
and we define the \emph{contravariant} functional $\mathscr{E}_{\rm con}: L^2(\Lambda^kT^*\!M)\to[0,+\infty]$ by 
\begin{align*}
\mathscr{E}_{\rm con}(\omega):=\left\{\begin{array}{cc}\displaystyle{\int_M \left[|\d\omega|^2+|\d_*\omega|^2 \right]\d\m} & \text{ if }\omega\in W^{1,2}(\Lambda^kT^*\!M) \\\infty & \text{otherwise.} \end{array}\right.
\end{align*}
}
\end{defn}
Arguing as for \cite[Theorems~5.3, 6.3 and 7.5]{Braun:Tamed2021}, $W^{1,2}(\Lambda^kT^*\!M)$ becomes a separable Hilbert space with respect to $\|\cdot\|_{W^{1,2}(\Lambda^kT^*\!M)}$. moreover, the functional $\mathscr{E}_{\rm con}:L^2(\Lambda^kT^*\!M)\to[0,+\infty]$ is clearly lower semi continuous. 

Again by \cite[Theorem~7.5 and Lemma~7.15]{Braun:Tamed2021}, we have ${\rm Reg}(\Lambda^kT^*\!M)\subset W^{1,2}(\Lambda^kT^*\!M)$, so that the following definition makes sense. 

\begin{defn}[The space $H^{1,2}(\Lambda^kT^*\!M)$]
{\rm The space $H^{1,2}(\Lambda^kT^*\!M)\subset W^{1,2}(\Lambda^kT^*\!M)$ is defined by 
the closure of ${\rm Reg}(\Lambda^kT^*\!M)$ with respect to $\|\cdot\|_{W^{1,2}(\Lambda^kT^*\!M)}$: 
\begin{align*}
H^{1,2}(\Lambda^kT^*\!M):=\overline{{\rm Reg}(\Lambda^kT^*\!M)}^{\|\cdot\|_{W^{1,2}(\Lambda^kT^*\!M)}}.
\end{align*}

}
\end{defn}

\subsection{Hodge-Kodaira Laplacian}\label{subsec:HodgeKodairaLaplacian}
We now define the Hosdge-Kodaira Laplacian in L$^2$-sense: 
\begin{defn}[$L^2$-Hodge-Kodaira Laplacian $\DD_k$]
{\rm The space $D(\DD_k)$ is defined to consist of all $\omega\in H^{1,2}(\Lambda^kT^*\!M)$ for which there exists $\alpha\in L^2(\Lambda^kT^*\!M)$ such that for every $\eta\in  H^{1,2}(\Lambda^kT^*\!M)$,
\begin{align*}
\int_M \langle \alpha,\eta\rangle \d\m=-\int_M \left[\langle \d\omega,\d\eta\rangle +\langle \d_*\omega,\d_*\eta\rangle  \right]\d\m.
\end{align*}
In case of existence, the element $\alpha$ is unique, denoted by $\DD_k\omega$ and called the 
\emph{Hodge Laplacian}, \emph{Hodge-Kodaira Laplacian} or \emph{Hodge-de~\!\!Rham Laplacian} of $\omega$. 
Formally $\DD_k\omega$ can  be written \lq\lq$\DD_k\omega=-(\d\d_*+\d_*\d)\omega$\rq\rq. 
}
\end{defn}

For the most important case $k=1$, we write $\DD$ instead of $\DD_1$. We see $\DD_0=\Delta$ the usual $L^2$-generator associated to the given quasi-regular strongly local Dirichlet form $(\mathscr{E},D(\mathscr{E}))$. Moreover, the Hodge-Kodaira Laplacian $\DD_k$ is a closed operator. 

\subsection{Heat flow of $1$-forms associated with $\DD$}\label{subsec:HFHosgekodaira}
We define the heat flow $P_t^{\rm HK}$ on $1$-forms associated to the functional $\widetilde{\mathscr{E}}_{\rm con}:L^2(T^*\!M)\to [0,+\infty]$ with 
\begin{align*}
\widetilde{\mathscr{E}}_{\rm con}(\omega):=\left\{\begin{array}{cc}\displaystyle{\int_M \left[|\d\omega|^2+|\d_*\omega|^2 \right]\d\m} & \text{ if }\omega\in H^{1,2}(T^*\!M), \\ \infty & \text{ otherwise.}\end{array}\right.
\end{align*}
We can write $\mathscr{E}^{\rm HK}$ instead of $\widetilde{\mathscr{E}}_{\rm con}$. Let $(P_t^{\rm HK})_{t\geq0}$ be the heat semigroup of bounded linear and self-adjoint operator on $L^2(T^*\!M)$ formally written by 
\begin{align*}
\text{\lq\lq $P_t^{\rm HK}:=e^{t\,\DD}$\rq\rq}.
\end{align*}
The following are important: 
\begin{lem}[{\cite[Lemma~8.32, Corollaries~8.35 and 8.36]{Braun:Tamed2021}}]
We have the following:
\begin{enumerate}
\item[\rm(1)]  For every $f\in D(\mathscr{E})$ and every $t>0$, $\d P_tf\in D(\DD)$ and 
\begin{align}
P_t^{\rm HK}\d f=\d P_tf.\label{eq:intertwining1}
\end{align}
\item[\rm(2)] If $\omega\in D(\d_*)$ and $t>0$, then $P_t^{\rm HK}\omega\in D(\d_*)$ and 
\begin{align}
\d_*P_t^{\rm HK}\omega=P_t\d_*\omega.\label{eq:intertwining2}
\end{align}
\item[\rm(3)] $\inf\sigma(-\Delta^{\kappa})\leq\inf \sigma(-\DD)$.
\end{enumerate}
\end{lem}
The formulas \eqref{eq:intertwining1} and \eqref{eq:intertwining2} are called \emph{intertwining properties}, which play 
a crucial role to prove the $L^p$-boundedness of Riesz operator.

\section{Ricci curvature measures}\label{sec:RicciCurvMeasure}
To state the notion of Ricci curvature measure, we need preliminary preparation.

The following are shown in \cite[Lemmas~8.1 and 8.2]{Braun:Tamed2021}:
\begin{lem}[{\cite[Lemma~8.1]{Braun:Tamed2021}}]
For every $g\in {\rm Test}(M)\cup\R\1_M$ and every $f\in{\rm Test}(M)$. we have $g\,\d f\in D(\DD)$ with 
\begin{align*}
\DD(g\,\d f)=g\,\d \Delta f+\Delta g\,\d f+2{\rm Hess}\,f(\nabla g,\cdot), 
\end{align*}
with the usual interpretations $\nabla \1_M=0$ and $\Delta \1_M=0$. More generally, for every $X\in {\rm Reg}(TM)$ and every $h\in {\rm Test}(M)\cup\R\1_M$, we have $hX^{\flat}\in D(\DD)$ with  
\begin{align*}
\DD(hX^{\flat})=h\DD X^{\flat}-\DD hX^{\flat}+2(\nabla_{\nabla h}X)^{\flat}.
\end{align*}
\end{lem}

\begin{lem}[{\cite[Lemma~8.2]{Braun:Tamed2021}}]\label{lem:Lemma8.2}
For every $X\in {\rm Reg}(TM)$, we have $|X|^2\in D({\text{\boldmath$\Delta$}}^{2\kappa})$ with 
\begin{align*}
{\text{\boldmath$\Delta$}}^{2\kappa}\frac{|X|^2}{2}\geq\left[|\nabla X|_{\rm HS}^2+\langle X,(\DD X^{\flat})^{\sharp}\rangle  \right]\m.
\end{align*}
\end{lem} 

\begin{defn}[The space $H^{1,2}_{\sharp}(TM)$]
{\rm We define the space $H^{1,2}_{\sharp}(TM)$ as the image of $H^{1,2}(T^*\!M)$ under the map $\sharp$, endowed with the norm 
\begin{align*}
\|X\|_{H^{1,2}_{\sharp}(TM)}:=\|X^{\flat}\|_{H^{1,2}(T^*\!M)}.
\end{align*}
}
\end{defn}
The following is proved in \cite[Lemma~8.8]{Braun:Tamed2021}:
\begin{lem}[{\cite[Lemma~8.8]{Braun:Tamed2021}}]\label{lem:Lemma8.8}
$H^{1,2}_{\sharp}(TM)$ is a subspace of $H^{1,2}(TM)$. The aforementioned natural inclusion is continuous. Additionally, for every $X\in H^{1,2}_{\sharp}(TM)$, 
\begin{align}
\mathscr{E}_{\rm cov}(X)\leq\mathscr{E}_{\rm con}(X^{\flat})-\langle \kappa, \wt{|X|}^2\rangle.\label{eq:ContBochHK}
\end{align}
\end{lem}
Then, we can state the main result of \cite{Braun:Tamed2021}: 
\begin{thm}[{{\cite[Theorem~8.9]{Braun:Tamed2021}}}]
\label{thm:Theorem8.9}
There exists a unique continuous mapping ${\bf Ric}^{\kappa}: H^{1,2}_{\sharp}(TM)^2\to \mathfrak{M}_{\rm f}^{\pm}(M)_{\mathscr{E}}$ satisfying the identity
\begin{align}
{\bf Ric}^{\kappa}(X,Y)={\text{\boldmath$\Delta$}}^{2\kappa}\frac{\langle X,Y\rangle }{2}-\left[\frac12\langle X,(\DD Y^{\flat})^{\sharp}\rangle 
+\frac12\langle Y,(\DD X^{\flat})^{\sharp}\rangle +\langle \nabla X\,|\,\nabla Y\rangle 
 \right]\m\label{eq:kappaRicciMeasure}
\end{align}
for every $X,Y\in {\rm Reg}(TM)$. Here $\mathfrak{M}_{\rm f}^{\pm}(M)_{\mathscr{E}}$ denotes the family of finite signed smooth measures. The map ${\bf Ric}^{\kappa}$ is symmetric and $\R$-linear. Furthermore, for every $X,Y\in H^{1,2}_{\sharp}(TM)$, it obeys
\begin{align}
{\bf Ric}^{\kappa}(X,X)&\geq0,\label{eq:NonKappaRicci}\\
{\bf Ric}^{\kappa}(X,Y)(M)&=\int_M\left[\langle \d X^{\flat},\d Y^{\flat}\rangle +\langle \d_*X^{\flat},\d_*Y^{\flat}\rangle  \right]\d\m\notag\\
 &\hspace{2cm}- \int_M\langle \nabla X\,|\,\nabla Y\rangle \d\m-\langle \kappa,\widetilde{\langle X,Y\rangle }\rangle ,\label{eq:RicciKappaMeasureTotal}\\
 \|{\bf Ric}^{\kappa}(X,Y)\|_{\rm TV}^2&\leq\left[\mathscr{E}_{\rm con}(X^{\flat})-\langle \kappa,\wt{|X|}^2\rangle  \right]\left[\mathscr{E}_{\rm con}(Y^{\flat})-\langle \kappa,\wt{|Y|}^2\rangle  \right].\label{eq:RicciKappaCont}
\end{align}
\end{thm}

Keeping in mind Theorem~\ref{thm:Theorem8.9}, we define the map 
${\bf Ric}:H^{1,2}_{\sharp}(TM)^2\to \mathfrak{M}_{\rm f}^{\pm}(M)_{\mathscr{E}}$ by 
\begin{align}
{\bf Ric}(X,Y):={\bf Ric}^{\kappa}(X,Y)+\langle \kappa,\widetilde{\langle X,Y\rangle }\rangle. \label{eq:RicciMeasure}
\end{align}
This map is well-defined, symmetric and $\R$-bilinear. ${\bf Ric}$ is also jointly continuous, which follows from polarization, \cite[Corollary~6.14]{Braun:Tamed2021} and Lemma~\ref{lem:Lemma8.8}. Lastly, by Theorem~\ref{thm:Theorem8.9}, Lemma~\ref{lem:Lemma8.2} and \cite[Lemma~8.12]{Braun:Tamed2021}, for every $X,Y\in{\rm Reg}(TM)$, we have a similar expression 
with \eqref{eq:kappaRicciMeasure}: 
\begin{align}
{\bf Ric}(X,Y)={\text{\boldmath$\Delta$}}\frac{\langle X,Y\rangle }{2}-\left[\frac12\langle X,(\DD Y^{\flat})^{\sharp}\rangle 
+\frac12\langle Y,(\DD X^{\flat})^{\sharp}\rangle +\langle \nabla X\,|\,\nabla Y\rangle  \right]\m\label{eq:RicciMeasure}
\end{align}
\begin{lem}[{\cite[Lemma~8.15]{Braun:Tamed2021}}]
For every $X,Y\in H_{\sharp}^{1,2}(TM)$ and $f\in {\rm Test}(M)$, we have 
\begin{align}
{\bf Ric}(fX,Y)=\tilde{f}\,{\bf Ric}(X,Y),\label{eq:RicciDensity}
\end{align}
in particular, 
\begin{align}
{\bf Ric}^{\kappa}(fX,Y)=\tilde{f}\,{\bf Ric}^{\kappa}(X,Y).\label{eq:RicciDensityKappa}
\end{align}
\end{lem}

\section{Proof of Theorem~\ref{thm:HessShcraderUhlenbrock}}

First note that the following condition ${\bf (R_3)}$ in \cite[Corollary~4.2.3]{BH} is satisfied for tamed Dirichlet space. 
\begin{align*}
{\bf (R_3)}: \text{There exists }\mathscr{C}\subset D(\Delta)\text{ dense in }D(\mathscr{E})\text{ such that for all }f\in\mathscr{C}, f^2\in D(\Delta). 
\end{align*}
In fact, $\mathscr{C}={\rm Test}(M)$ satisfies ${\bf (R_3)}$. 
It is proved in \cite[Corollary~4.2.3]{BH} that ${\bf (R_3)}$ implies the following ${\bf (R_2)}$ and ${\bf (R_2)'}$.
\begin{align*}
{\bf (R_2)}:& \text{ For all }f\in D(\Delta), \,f^2\in D(\Delta_1).\\
{\bf (R_2)'}:&  \text{ There exists }\mathscr{C}\subset D(\Delta)\text{ dense in }D(\mathscr{E})\text{ such that for all }f\in\mathscr{C}, \,f^2\in D(\Delta_1).
\end{align*}

The proof of Theorem~\ref{thm:HessShcraderUhlenbrock} is split in several lemmas. 

\begin{lem}\label{eq:BochnerContractionIneq}
$X\in H^{1,2}(TM)$ implies $|X|\in D(\mathscr{E})$ and 
\begin{align}
\mathscr{E}(|X|,|X|)\leq \wt{\mathscr{E}}_{\rm cov}(X,X)<\infty. \label{eq:ContractionBochner1}
\end{align}
Moreover, for $f\in D(\mathscr{E})\cap L^{\infty}(M;\m)_+$ with $fX\in H^{1,2}(TM)$, we have $f|X|\in D(\mathscr{E})$ and 
\begin{align}
\mathscr{E}(|X|,f|X|)\leq \wt{\mathscr{E}}_{\rm cov}(X,fX). \label{eq:ContractionBochner2}
\end{align}
\end{lem}
\begin{proof}[{\bf Proof}]
The proof of \eqref{eq:ContractionBochner1}
can be directly deduced from Lemma~\ref{lem:KatoIneq}. 
Next we show \eqref{eq:ContractionBochner2}. Assume $fX\in H^{1,2}(TM)$ for $f\in D(\mathscr{E})\cap L^{\infty}(M;\m)_+$. By \eqref{eq:ContractionBochner1}, we have $f|X|\in D(\mathscr{E})$. Moreover, by Lemma~\ref{lem:bochnerHF}(2), 
\begin{align*}
((I- P_t)|X|,f|X|)_{L^2(M;\m)}\leq
((I- P_t^{\rm B})X,fX)_{L^2(TM)}.
\end{align*} 
Divided by $t>0$ and letting $t\to0$, we obtain \eqref{eq:ContractionBochner2}. 
\end{proof}
\begin{lem}\label{lem:HKEnergyCont}
Take $\omega\in H^{1,2}(T^*\!M)(=D(\mathscr{E}^{\rm HK}))$. Then $|\omega|\in D(\mathscr{E})$ and 
\begin{align}
\mathscr{E}^{\kappa}(|\omega|,|\omega|)\leq\mathscr{E}^{\rm HK}(\omega,\omega).\label{eq:KEnergyCont}
\end{align}
\end{lem}
\begin{proof}[{\bf Proof}] 
$\omega\in H^{1,2}(T^*\!M)$ implies $\omega^{\sharp}\in H^{1,2}_{\sharp}(TM)\subset H^{1,2}(TM)$ and $|\omega|=|\omega^{\sharp}|\in D(\mathscr{E})$. Then 
\begin{align*}
\mathscr{E}(|\omega|,|\omega|)&=\mathscr{E}(|\omega^{\sharp}|,|\omega^{\sharp}|)
\stackrel{\eqref{eq:ContractionBochner1}}{\leq}\wt{\mathscr{E}}_{\rm cov}(\omega^{\sharp},\omega^{\sharp})\stackrel{\eqref{eq:ContBochHK}}{
\leq}\mathscr{E}^{\rm HK}(\omega,\omega)-\langle \kappa,|\omega|^2\rangle,
\end{align*}
which implies the conclusion. 
\end{proof}

\begin{lem}\label{lem:HKEnergyCont*}
If $\omega\in H^{1,2}(T^*\!M)\cap L^{\infty}(T^*\!M)$ {\rm(}resp.~$\omega\in H^{1,2}(T^*\!M)${\rm)}
and $f\in D(\mathscr{E})\cap L^{\infty}(M;\m)$ {\rm(}resp.~$f\in {\rm Test}(M)${\rm)},  
then $f\omega\in H^{1,2}(T^*\!M)\cap L^{\infty}(T^*\!M)$ {\rm(}resp.~$f\omega\in H^{1,2}(T^*\!M)${\rm)}.   
\end{lem}
\begin{proof}[{\bf Proof}]
By \cite[Remark~6.10]{Braun:Tamed2021}, we have that 
$f\in D(\mathscr{E})\cap L^{\infty}(M;\m)$ (resp.~$f\in {\rm Test}(M)$)
 and $X\in H^{1,2}(TM)\cap L^{\infty}(TM)$ (resp.~$X\in H^{1,2}(TM)$) imply 
$fX\in H^{1,2}(TM)\cap L^{\infty}(TM)$ (resp.~$fX\in  H^{1,2}(TM)$)
and $\nabla(fX)=\nabla f\otimes X+f\nabla X$. Therefore, 
$f\in D(\mathscr{E})\cap L^{\infty}(M;\m)$ (resp.~$f\in \text{\rm Test}(M)$) 
and $\omega\in H^{1,2}(T^*\!M)\cap L^{\infty}(T^*\!M)$ (resp.~$\omega\in H^{1,2}(T^*\!M)$) imply 
$f\omega\in H^{1,2}(T^*\!M)\cap L^{\infty}(T^*\!M)$ (resp.~$f\omega\in H^{1,2}(T^*\!M)$), because 
$(f\omega)^{\sharp}=f\omega^{\sharp}$ holds in view of that for any $\eta\in H^{1,2}(T^*\!M)$
\begin{align*}
\langle (f\omega^{\sharp})^{\flat},\eta\rangle =\eta(f\omega^{\sharp})=\langle \eta^{\sharp},f\omega^{\sharp}\rangle 
&=f\langle \omega,\eta\rangle =\langle f\omega,\eta\rangle.
\end{align*}. 
\end{proof}

\begin{lem}\label{lem:HKEnergyCont**}
Take $\omega\in H^{1,2}(T^*\!M)$ 
and $f\in {\rm Test}(M)_+$.
Then $f|\omega|\in D(\mathscr{E}^{\kappa})=D(\mathscr{E})$ and 
\begin{align}
\mathscr{E}^{\kappa}(|\omega|,f|\omega|)\leq\mathscr{E}^{\rm HK}(\omega,f\omega).\label{eq:KEnergyCont**}
\end{align}
\end{lem}
\begin{proof}[{\bf Proof}]
Suppose $\omega\in H^{1,2}(T^*\!M)$ 
and $f\in {\rm Test}(M)_+$.
By Lemma~\ref{lem:HKEnergyCont*}, we have $f\omega\in H^{1,2}(T^*\!M)$, hence 
$f|\omega|\in D(\mathscr{E})$ by Lemma~\ref{lem:HKEnergyCont*}. 
On the other hand, by \eqref{eq:RicciKappaMeasureTotal}, 
 we have 
\begin{align}
{\bf Ric}(\omega^{\sharp},f\omega^{\sharp})(M)=\mathscr{E}^{\rm HK}(\omega,f\omega)-\wt{\mathscr{E}}_{\rm cov}(\omega^{\sharp},f\omega^{\sharp}).\label{eq:RicciIdentity}
\end{align}
By Lemma~\ref{eq:BochnerContractionIneq}, we then have
\begin{align*}
\mathscr{E}^{\kappa}(|\omega|,f|\omega|)&=\mathscr{E}(|\omega|,f|\omega|)+\langle \kappa,f|\omega|^2\rangle \\
&\leq \wt{\mathscr{E}}_{\rm cov}(\omega^{\sharp},f\omega^{\sharp})+\langle \kappa,f|\omega|^2\rangle \\
&\hspace{-0.2cm}\stackrel{\eqref{eq:RicciIdentity}}{=}\mathscr{E}^{\rm HK}(\omega,f\omega)-{\bf Ric}^{\kappa}(\omega^{\sharp},f\omega^{\sharp})(M)\\
&\hspace{-0.2cm}\stackrel{\eqref{eq:RicciDensityKappa}}{=}\mathscr{E}^{\rm HK}(\omega,f\omega)-\int_M \wt{f}{\rm d}{\bf Ric}^{\kappa}(\omega^{\sharp},\omega^{\sharp})\\
&\hspace{-0.2cm}\stackrel{\eqref{eq:NonKappaRicci}}{
\leq} \mathscr{E}^{\rm HK}(\omega,f\omega).
\end{align*}
\end{proof}

\begin{cor}\label{cor:HKEnergyCont**}
Take $\omega\in D(\DD)\cap L^{\infty}(T^*\!M)$ and $f\in D(\mathscr{E})\cap L^{\infty}(M;\m)_+$. 
Then $f\omega\in H^{1,2}(T^*\!M)$, $f|\omega|\in D(\mathscr{E})\cap L^{\infty}(M;\m)$, and \eqref{eq:KEnergyCont**} hold. 
\end{cor}
\begin{proof}[{\bf Proof}]
We already know $f\omega\in H^{1,2}(T^*\!M)$ by Lemma~\ref{lem:HKEnergyCont*}. 
By Lemma~\ref{lem:HKEnergyCont**}, we have $(P_tf)\omega\in H^{1,2}(T^*\!M)$ and $(P_tf)|\omega|\in D(\mathscr{E})$, and 
\begin{align*}
\mathscr{E}^{\kappa}(|\omega|,(P_tf)|\omega|)\leq \left(-\DD\omega, (P_tf)\omega\right)_{L^2(T^*\!M)}, 
\end{align*}
because $P_tf\in  {\rm Test}(M)_+$. Since $\{P_tf\}_{t\geq0}$ converges to $f$ in $D(\mathscr{E})$ as $t\to0$ 
and weakly* converges to $f$ in $L^{\infty}(M;\m)$ as $t\to0$, we can deduce \eqref{eq:KEnergyCont**}.  
\end{proof}

\begin{lem}\label{lem:HKEnergyCont***}
Take $\omega\in D(\DD)\cap L^{\infty}(T^*\!M)$.  
Then 
\begin{align}
\left(-\DD\omega, g\frac{\omega}{|\omega|} \right)_{L^2(T^*\!M)}\geq\mathscr{E}^{\kappa}(|\omega|,g)\quad \text{ for all }\quad g\in {\rm Test}(M)_+.
\label{eq:KEnergyCont***}
\end{align}
Here we set $\omega/|\omega|:=0$ if $\omega=0$.
\end{lem}
\begin{proof}[{\bf Proof}]
By Lemma~\ref{lem:HKEnergyCont}, we see $|\omega|\in D(\mathscr{E})\cap L^{\infty}(M;\m)$.   
For each $\eps>0$, we set  
$|\omega|_{\eps}:=\sqrt{|\omega|^2+\eps^2}$. 
Then we see $\frac{|\omega|}{|\omega|_{\eps}}\in D(\mathscr{E})\cap L^{\infty}(M;\m)$ (cf. \cite[Chapter I, Theorem~4.12 and Exercise~4.16]{MR}). We may assume $|\omega|$ and $|\omega|_{\eps}$ are $\mathscr{E}$-quasi continuous.  

 For $g\in {\rm Test}(M)_+$, 
 we consider $f:=g/|\omega|_{\eps}\in D(\mathscr{E})\cap L^{\infty}(M;\m)_+$.
We apply Corollary~\ref{cor:HKEnergyCont**} for $f\in D(\mathscr{E})\cap L^{\infty}(M;\m)_+$ 
 so that $f\omega\in H^{1,2}(T^*\!M)$, $f|\omega|\in D(\mathscr{E})$ and 
\begin{align*}
\left(-\DD\omega, g\frac{\omega}{|\omega|_{\ep}} \right)_{L^2(T^*\!M)}&=
\left(-\DD\omega, f\omega \right)_{L^2(T^*\!M)}\\
&=\mathscr{E}^{\rm HK}(\omega, f\omega)\\
&\hspace{-0.2cm}\stackrel{\eqref{eq:KEnergyCont**}}{\geq}\mathscr{E}^{\kappa}(|\omega|,f|\omega|)\\
&=\mathscr{E}^{\kappa}\left(|\omega|,g\frac{|\omega|}{|\omega|_{\eps}}\right)\\
&=\int_M\frac{g}{|\omega|_{\eps}^3}(|\omega|_{\eps}^2-|\omega|^2)\Gamma(|\omega|,|\omega|)\d\m\\
&\hspace{1cm}+\int_M\frac{|\omega|}{|\omega|_{\eps}}{\Gamma(|\omega|,g)\d\m+
\left\langle \kappa,g\frac{|\omega|^2}{|\omega|_{\eps}}\right\rangle}\\
&\geq \int_M\frac{|\omega|}{|\omega|_{\eps}}{\Gamma(|\omega|,g)\d\m+
\left\langle \kappa,g\frac{|\omega|^2}{|\omega|_{\eps}}\right\rangle}.
\end{align*}  
Now letting $\eps\to0$, we have
\begin{align*}
\int_M\frac{|\omega|}{|\omega|_{\eps}}\Gamma(|\omega|,g)\d\m\to \int_M\1_{\{\omega\ne0\}}\Gamma(|\omega|,g)\d\m.
\end{align*}
On the other hand, 
\begin{align*}
\Gamma(|\omega|^{1+\eps},|\omega|^{1+\eps})=(1+\eps)^2|\omega|^{2\eps}\Gamma(|\omega|,|\omega|).
\end{align*}
Hence
\begin{align*}
\1_{\{\omega=0\}}\Gamma(|\omega|^{1+\eps},|\omega|^{1+\eps})=0.
\end{align*}
Letting $\eps\to0$, we see 
\begin{align*}
\1_{\{\omega=0\}}\Gamma(|\omega|,|\omega|)=0,
\end{align*}
in particular, 
\begin{align*}
\1_{\{\omega=0\}}\Gamma(|\omega|,g)=0.
\end{align*}
Now we eventually have 
\begin{align*}
\left(-\DD\omega, g\frac{\omega}{|\omega|} \right)_{L^2(T^*\!M)}
&\geq \int_M\Gamma(|\omega|,g)\d\m+\langle \kappa,g|\omega|\rangle \\
&=\mathscr{E}(|\omega|,g)+\langle \kappa,g|\omega|\rangle \\
&=\mathscr{E}^{\kappa}(|\omega|,g)
\end{align*}
for $g\in {\rm Test}(M)_+$. 
\end{proof}

\begin{lem}\label{lem:continuityofHKsemigrouop}
Suppose $\omega\in H^{1,2}(T^*\!M)$. Then 
\begin{align}
\int_0^t\int_M\wt{|P_s^{\rm HK}\omega|}^2\d\kappa^-\d s<\infty.\label{eq:continuityofHKsemigrouop}
\end{align}
\end{lem}
\begin{proof}[\bf Proof]
Under $\omega\in H^{1,2}(T^*\!M)$, \eqref{eq:continuityofHKsemigrouop} follows from the continuity of 
\begin{align}
[0,t]\ni s\mapsto \int_M\wt{|P_s^{\rm HK}\omega|}^2\d\kappa^-<\infty.\label{eq:finiteness}
\end{align}
The continuity \eqref{eq:finiteness} can be confirmed by combining \cite[Corollary~6.14, (8.5) and 
Theorem~8.31(v)]{Braun:Tamed2021}. 
\end{proof}

\begin{lem}\label{lem:HessSchradarUhlenbrock2}
For every $\omega\in L^2(T^*\!M)$ and $t>0$, we have 
\begin{align}
|P_t^{\rm HK}\omega|^2\leq P_t^{-2\kappa^-}|\omega|^2\quad \m\text{-a.e.}\label{eq:HessSchradarUhlenbrock2}
\end{align}
In particular, for $\omega\in L^2(T^*\!M)\cap L^{\infty}(T^*\!M)$ and $\alpha> C_{\kappa}$, 
\begin{align}
\|R_{\alpha}^{\rm HK}\omega\|_{L^{\infty}(T^*\!M)}\leq \frac{\sqrt{C(\kappa)}}{\alpha-C_{\kappa}}\|\omega\|_{L^{\infty}(T^*\!M)}.\label{eq:HessSchradarUhlenbrock3}
\end{align}
Here $C(\kappa)$ and $C_{\kappa}$ are the positive constants appeared in  \eqref{eq:KatoContraction}. 
\end{lem}
\begin{proof}[\bf Proof]
We may assume $\kappa^+=0$, because ${\sf BE}_2(-\kappa^-,\infty)$ is satisfied. 
Since $H^{1,2}(T^*\!M)$ is dense in $L^2(T^*\!M)$ and 
$\|P_t^{\rm HK}\omega\|_{L^2(T^*\!M)}\leq \|\omega\|_{L^2(T^*\!M)}$ for $\omega\in L^2(T^*\!M)$ (see \cite[Theorem~8.31(iii)]{Braun:Tamed2021}), if \eqref{eq:HessSchradarUhlenbrock2} holds for 
$\omega\in H^{1,2}(T^*\!M)$, 
then it holds for $\omega\in L^2(T^*\!M)$. So we may assume $\omega\in H^{1,2}(T^*\!M)$. 
Take $g\in L^2(M;\m)\cap L^{\infty}(M;\m)\cap \mathscr{B}_+(M)$ and set $g_{\alpha}:=R_{\alpha}g=\E_{\cdot}[\int_0^{\infty}e^{-\alpha s}g(X_s)\d s]$.  
Note here that $(G_{\alpha})_{\alpha\geq0}$ defined by $G_{\alpha}f:=\int_0^{\infty}e^{-\alpha t}P_tf\d t$ for $f\in L^1(M;\m)$ forms a strongly continuous resolvent on $L^1(M;\m)$. Since 
$R_{\alpha}f(x):=\E_x\left[\int_0^{\infty}e^{-\alpha t}f(X_t)\d t\right]$ is 
$\alpha$-excessive for $f\in L^1(M;\m)\cap \mathscr{B}(M)_+$, it is 
an $\mathscr{E}$-q.e.~finite and $\mathscr{E}$-quasi continuous function for $f\in L^1(M;\m)\cap \mathscr{B}(M)_+$ in view of  
\cite[Lemma~4.4]{Kw:maximumprinciple} and \cite[Theorem~4.6.1]{FOT}.  
The content of \cite[Theorem~4.6.1]{FOT} remains valid for quasi-regular Dirichlet forms. 
We may assume $\wt{|P_s^{\rm HK}\omega|}^2\in L^1(M;\m)\cap \mathscr{B}(M)_+$, where 
$\wt{|P_s^{\rm HK}\omega|}$ is an $\mathscr{E}$-quasi continuous $\m$-version of $|P_s^{\rm HK}\omega|\in D(\mathscr{E})$.   

We now set a function $F_n:[0,t]\to\R$ defined by 
\begin{align*}
F_n(s):=\int_MP_{t-s}^{2\kappa}g_{\alpha}\,n P_{\frac{1}{n}}G_n|P_s^{\rm HK}\omega|^2\d \m.
\end{align*}
Note that $|P_s^{\rm HK}\omega|^2\in\mathscr{G}_{\rm reg}\subset \mathscr{G}$ by \cite[Proposition~6.11]{Braun:Tamed2021}, 
where $\mathscr{G}$ and $\mathscr{G}_{\rm reg}$ are the first order Sobolev spaces defined in 
\cite[Definitions~6.18 and 5.19]{Braun:Tamed2021}.
Hence
$P_{\frac{1}{n}}G_n|P_s^{\rm HK}\omega|^2\in L^1(M;\m)$, and $\Delta^{2\kappa}P_{t-s}^{2\kappa}g_{\alpha} \in L^{\infty}(M;\m)$. Since $(P_t^{2\kappa})_{t\geq0}$ is weakly* continuous semigroup on $L^{\infty}(M;\m)$, we can deduce  
\begin{align*}
\frac{1}{n}F'_n(s)&=\int_M\left(-\Delta^{2\kappa}P_{t-s}^{2\kappa}g_{\alpha}\, P_{\frac{1}{n}}G_n|P_s^{\rm HK}\omega|^2+
2(P_{t-s}^{2\kappa}g_{\alpha}) P_{\frac{1}{n}}G_n \langle P_s^{\rm HK}\omega,\DD P_s^{\rm HK}\omega\rangle  \right)\d\m\\
&=\int_M\left(-\Delta^{2\kappa}P_{t-s}^{2\kappa}g_{\alpha} \right)P_{\frac{1}{n}}G_n|P_s^{\rm HK}\omega|^2\d\m
+2\int_M\langle 
(P_{\frac{1}{n}}G_n P_{t-s}^{2\kappa}g_{\alpha}) P_s^{\rm HK}\omega,\DD P_s^{\rm HK}\omega\rangle\d\m.
\end{align*}
The first term of the right hand side in the last line is 
\begin{align*}
\int_M&\left(-\Delta^{2\kappa}P_{t-s}^{2\kappa}g_{\alpha} \right)\, P_{\frac{1}{n}}G_n|P_s^{\rm HK}\omega|^2\d\m\\
&=
\lim_{k\to\infty}\int_M\left(-\Delta^{2\kappa}P_{t-s}^{2\kappa}g_{\alpha} \right)\, P_{\frac{1}{n}}G_n\left(
|P_s^{\rm HK}\omega|^2\land k^2\right)\d\m\\
&=\lim_{k\to\infty}\mathscr{E}^{2\kappa}(P_{t-s}^{2\kappa}g_{\alpha},P_{\frac{1}{n}}G_n(|P_s^{\rm HK}\omega|^2\land k^2) )\\
&=\lim_{k\to\infty}\left\{\mathscr{E}\left(P_{t-s}^{2\kappa}g_{\alpha},P_{\frac{1}{n}}G_n(|P_s^{\rm HK}\omega|^2\land k^2) \right)-
2\int_M p_{t-s}^{2\kappa}g_{\alpha}\;p_{\frac{1}{n}}R_n(\wt{|P_s^{\rm HK}\omega|}^2\land k^2)\d\kappa^-\right\}\\
&=\lim_{k\to\infty}\int_M\langle 
\nabla P_{\frac{1}{n}}G_n P_{t-s}^{2\kappa}g_{\alpha},\nabla (|P_s^{\rm HK}\omega|^2\land k^2)
\rangle\d\m -2\int_M (p_{t-s}^{2\kappa}g_{\alpha})\,p_{\frac{1}{n}}R_n \wt{|P_s^{\rm HK}\omega|}^2\d\kappa^-
\\
&=\int_M\langle 
\nabla P_{\frac{1}{n}}G_n P_{t-s}^{2\kappa}g_{\alpha},\nabla |P_s^{\rm HK}\omega|^2
\rangle\d\m -2\int_M (p_{t-s}^{2\kappa}g_{\alpha})\,p_{\frac{1}{n}}R_n\wt{|P_s^{\rm HK}\omega|}^2\d\kappa^-,
\end{align*}
where we use the fact that $\nabla (|P_s^{\rm HK}\omega|^2\land k^2)\to \nabla |P_s^{\rm HK}\omega|^2$ as $k\to\infty$ in $L^1(TM)$ and 
$P_{\frac{1}{n}}G_n P_{t-s}^{2\kappa}g\in{\rm Test}(M)_+$
 by Lemma~\ref{lem:boundedEst}, hence $|\nabla P_{\frac{1}{n}}G_n P_{t-s}^{2\kappa}g|\in L^{\infty}(M;\m)$.
Moreover, $p_{t-s}^{2\kappa}g_{\alpha}$ is an $\mathscr{E}$-quasi continuous $\m$-version of $P_{t-s}^{2\kappa}g_{\alpha}$ (see \cite[Theorem~7.3]{Kw:stochII}).   
Therefore, we have 
\begin{align*}
\frac{1}{n}F'_{n}(s)&=\int_M\langle 
\nabla P_{\frac{1}{n}}G_n P_{t-s}^{2\kappa}g_{\alpha},\nabla |P_s^{\rm HK}\omega|^2
\rangle\d\m -2\int_M (p_{t-s}^{2\kappa}g_{\alpha})\,p_{\frac{1}{n}}R_n \wt{|P_s^{\rm HK}\omega|}^2\d\kappa^-\\
&\hspace{2cm}+
2\int_M\langle 
P_{\frac{1}{n}}G_n P_{t-s}^{2\kappa}g_{\alpha}\,P_s^{\rm HK}\omega,\DD P_s^{\rm HK}\omega\rangle\d\m\\
&=2\int_M\langle \nabla(P_s^{\rm HK}\omega)^{\sharp}\,|\,\nabla P_{\frac{1}{n}}G_n P_{t-s}^{2\kappa}g_{\alpha}\otimes (P_s^{\rm HK}\omega)^{\sharp}\rangle_{\rm HS}\d\m\\
&\hspace{2cm}-2\int_M
\langle \d(P_{\frac{1}{n}}G_n P_{t-s}^{2\kappa}g_{\alpha}P_s^{\rm HK}\omega),\d P_s^{\rm HK}\omega\rangle\d\m \\
&\hspace{4cm}
-2\int_M\langle \d_*(P_{\frac{1}{n}}G_n P_{t-s}^{2\kappa}g_{\alpha} P_s^{\rm HK}\omega),\d_* P_s^{\rm HK}\omega\rangle\d\m\\
&\hspace{6cm}
-2\int_M (p_{t-s}^{2\kappa}g_{\alpha})\,p_{\frac{1}{n}}R_n \wt{|P_s^{\rm HK}\omega|}^2\d\kappa^-.
\end{align*}
Since 
\begin{align*}
\langle \nabla(P_s^{\rm HK}\omega)^{\sharp}\,&|\,\nabla P_{\frac{1}{n}}G_n P_{t-s}^{2\kappa}g_{\alpha}\otimes (P_s^{\rm HK}\omega)^{\sharp}\rangle_{\rm HS}\\&=
\langle\nabla (P_s^{\rm HK}\omega)^{\sharp}\,|\,\nabla (p_{\frac{1}{n}}R_n P_{t-s}^{2\kappa}g_{\alpha})P_s^{\rm HK}\omega)^{\sharp} \rangle_{\rm HS}
-P_{\frac{1}{n}}G_n P_{t-s}^{2\kappa}g_{\alpha}|\nabla(P_s^{\rm HK}\omega)^{\sharp}|_{\rm HS}^2\\
&\leq \langle\nabla (P_s^{\rm HK}\omega)^{\sharp}\,|\,\nabla (P_{\frac{1}{n}}G_n P_{t-s}^{2\kappa}g_{\alpha})P_s^{\rm HK}\omega)^{\sharp} \rangle_{\rm HS},
\end{align*}
we have 
\begin{align}
\frac{1}{n}F'_n(s)&\leq  2\int_M\langle\nabla (P_s^{\rm HK}\omega)^{\sharp}\,|\,\nabla ( P_{\frac{1}{n}}G_n P_{t-s}^{2\kappa}g_{\alpha})P_s^{\rm HK}\omega)^{\sharp} \rangle_{\rm HS}\d\m\notag\\
&\hspace{2cm}-2\int_M
\langle \d(P_{\frac{1}{n}}G_n P_{t-s}^{2\kappa}g_{\alpha} P_s^{\rm HK}\omega),\d P_s^{\rm HK}\omega\rangle\d\m \notag\\
&\hspace{4cm}
-2\int_M\langle \d_*(P_{\frac{1}{n}}G_n P_{t-s}^{2\kappa}g_{\alpha} P_s^{\rm HK}\omega),\d_* P_s^{\rm HK}\omega\rangle\d\m\notag\\
&\hspace{6cm}
-2\int_M (p_{t-s}^{2\kappa}g_{\alpha})\,p_{\frac{1}{n}}R_n \wt{|P_s^{\rm HK}\omega|}^2\d\kappa^-\notag\\
&= -2\int_Mp_{\frac{1}{n}}R_n p_{t-s}^{2\kappa}g_{\alpha}\, \d{\bf Ric}((P_s^{\rm HK}\omega)^{\sharp},(P_s^{\rm HK}\omega)^{\sharp})
-2\int_M (p_{t-s}^{2\kappa}g_{\alpha})\,p_{\frac{1}{n}}R_n \wt{|P_s^{\rm HK}\omega|}^2\d\kappa^-\notag
\end{align}
\begin{align}
&= -2\int_Mp_{\frac{1}{n}}R_n p_{t-s}^{2\kappa}g_{\alpha}\, \d{\bf Ric}^{\kappa}((P_s^{\rm HK}\omega)^{\sharp},(P_s^{\rm HK}\omega)^{\sharp})\notag\\
&\hspace{2cm}+ 
2\int_Mp_{\frac{1}{n}}R_n p_{t-s}^{2\kappa}g_{\alpha}\,\wt{|P_s^{\rm HK}\omega|}^2\d\kappa^--2\int_M
(p_{t-s}^{2\kappa}g_{\alpha})\,p_{\frac{1}{n}}R_n \wt{|P_s^{\rm HK}\omega|}^2\d\kappa^-\notag\\
&\hspace{-0.2cm}\stackrel{\eqref{eq:NonKappaRicci}}{\leq}2\int_Mp_{\frac{1}{n}}R_n p_{t-s}^{2\kappa}g_{\alpha}\,\wt{|P_s^{\rm HK}\omega|}^2\d\kappa^--2\int_M
(p_{t-s}^{2\kappa}g_{\alpha})\,p_{\frac{1}{n}}R_n \wt{|P_s^{\rm HK}\omega|}^2\d\kappa^-.\label{eq:Dominates}
\end{align}
In the first equality above, we use $P_{\frac{1}{n}}G_n P_{t-s}^{2\kappa}g_{\alpha}\in {\rm Test}(M)_+$ again 
and \cite[Lemma~8.14]{Braun:Tamed2021}. 
Thanks to the right continuity of   
$s\mapsto g_{\alpha}(X_s)$ under the law $\P_x$ and  
the Lebesgue's dominated convergence theorem, 
we can deduce that 
$n p_{\frac{1}{n}}R_n p_{t-s}^{2\kappa}g_{\alpha}\to p_{t-s}^{2\kappa}g_{\alpha}$ as $n\to\infty$ $\mathscr{E}$-q.e.  
From this, \eqref{eq:Dominates} and the boundedness of $P_{t-s}^{2\kappa}g_{\alpha}$, we have 
\begin{align}
\hspace{-0.5cm}\varlimsup_{n\to\infty}F'_n(s)&\leq 
2\int_M p_{t-s}^{2\kappa}g_{\alpha}\,\wt{|P_s^{\rm HK}\omega|}^2\d\kappa^--\varliminf_{n\to\infty}2\int_M
p_{t-s}^{2\kappa}g_{\alpha}\; np_{\frac{1}{n}}R_n \wt{|P_s^{\rm HK}\omega|}^2\d\kappa^-\leq0,\label{eq:limsupleq0}
\end{align}
where we use the Fatou's lemma and dominated convergence theorem for 
\begin{align*}
\int_M p_{t-s}^{2\kappa}g_{\alpha}\,\wt{|P_s^{\rm HK}\omega|}^2\d\kappa^-=\lim_{n\to\infty}
\int_Mnp_{\frac{1}{n}}R_n P_{t-s}^{2\kappa}g_{\alpha}\,\wt{|P_s^{\rm HK}\omega|}^2\d\kappa^-
\end{align*}
under \eqref{eq:finiteness}. 
Recall that $p_{\frac{1}{n}}R_n\wt{|P_s^{\rm HK}\omega|}^2=R_np_{\frac{1}{n}}\wt{|P_s^{\rm HK}\omega|}^2\in L^1(M;\m)$ 
is  $\mathscr{E}$-q.e.~finite and $\mathscr{E}$-quasi continuous. 
Integrating \eqref{eq:limsupleq0} from $0$ to $t$ and multiplying $(-1)$, we have
\begin{align*}
\int_0^t\sup_{n\geq1}f_n(s)\d s\geq0,
\end{align*}
where $f_n(s):=\inf_{\ell\geq n}\left(-F_{\ell}'(s) \right)$. 
We see $f_1(s)\leq f_n(s)\leq f_{n+1}(s)$ for all $s\in[0,t]$ and $n\in\mathbb{N}$. 
We can deduce
$\int_0^tf_1(s)\d s>-\infty$. Indeed, 
 \begin{align*}
 -\int_0^t f_1(s)\d s&\leq \int_0^t\sup_{\ell\geq1}F'_{\ell}(s)\d s\\
 &\hspace{-0.2cm}\stackrel{\eqref{eq:Dominates}}{\leq}\int_0^t\sup_{\ell\geq1}
 \left(
 \int_M\ell R_{\ell}p_{\frac{1}{\ell}}p_{t-s}^{2\kappa}g_{\alpha}\;\wt{|P_s^{\rm HK}\omega|}^2\d\kappa^-
 -\int_Mp_{t-s}^{2\kappa}g_{\alpha}\; \ell R_{\ell}p_{\frac{1}{\ell}}\wt{|P_s^{\rm HK}\omega|}^2\d\kappa^-
 \right)\d s\\
 &\leq \int_0^t\|P_{t-s}^{2\kappa}g_{\alpha}\|_{\infty}\int_M\wt{|P_s^{\rm HK}\omega|}^2\d\kappa^-\d s\\
 &\leq C(2\kappa)e^{C_{2\kappa}t}\|g_{\alpha}\|_{\infty}\int_0^t\int_M\wt{|P_s^{\rm HK}\omega|}^2\d\kappa^-\d s<\infty
 \end{align*}
 by \eqref{eq:continuityofHKsemigrouop} under $\omega\in  H^{1,2}(T^*\!M)$. By a variant of Beppo Levi's monotone convergence theorem under 
 $\int_0^tf_1(s)\d s>-\infty$, we get
 \begin{align*}
 -\varlimsup_{n\to\infty}\int_0^t F_n'(s)\d s\geq
 \lim_{n\to\infty}\int_0^t f_n(s)\d s=\int_0^t\sup_{n\geq1}f_n(s)\d s\geq0.
 \end{align*} 
 Thus, 
 \begin{align*}
 \varlimsup_{n\to\infty}\left(\int_M g_{\alpha}\, nP_{\frac{1}{n}}G_n |P_t^{\rm HK}\omega|^2\d\m-\int_M
 (P_{t-s}^{2\kappa}g_{\alpha})\, nP_{\frac{1}{n}}G_n |\omega|^2\d\m \right)&\leq 
 \varlimsup_{n\to\infty}\int_0^t F_n'(s)\d s\leq0.
 \end{align*}
Since $\lim_{n\to\infty}\|P_{\frac{1}{n}}nG_nf-f\|_{L^1(M;\m)}=0$ for $f\in L^1(M;\m)$, we have 
\begin{align*}
\int_M g_{\alpha}|P_t^{\rm HK}\omega|^2\d\m\leq \int_M (P_t^{2\kappa}g_{\alpha}) |\omega|^2\d\m
=\int_M g_{\alpha} P_t^{2\kappa}|\omega|^2\d\m.
\end{align*}
Since $\alpha g_{\alpha}=\alpha R_{\alpha}g\in L^{\infty}(M;\m)$ weakly* converges to $g$ in  $L^{\infty}(M;\m)$ as $\alpha\to\infty$ and $g\in L^2(M;\m)\cap L^{\infty}(M;\m)\cap \mathscr{B}_+(M)$ is arbitrary, we obtain \eqref{eq:HessSchradarUhlenbrock2}. 

Next we prove \eqref{eq:HessSchradarUhlenbrock3}. Under Assumption~\ref{asmp:Tamed}, from \eqref{eq:KatoContraction}, 
we can get
\begin{align}
\|P_t^{2\kappa}\|_{\infty,\infty}\leq C(\kappa)e^{2C_{\kappa}t}.\label{eq:2ExtendedKato}
\end{align}
That is, under $2\kappa^-\in S_{E\!K}({\bf X})$, we have \eqref{eq:2ExtendedKato}, because 
$C_{2\kappa}\leq C_{\kappa}\leq 2C_{\kappa}$ and $C(2\kappa)=C(\kappa)$ from 
\eqref{eq:KatoCoincidence}. Then
\begin{align*}
\|P_t^{\rm HK}\omega\|_{L^{\infty}(T^*\!M)}&\leq \sqrt{P_t^{2\kappa}1}\|\omega\|_{L^{\infty}(T^*\!M)}\\
&\leq 
\sqrt{C(\kappa)e^{2C_{\kappa}t}}\|\omega\|_{L^{\infty}(T^*\!M)}\\
&=\sqrt{C(\kappa)}e^{C_{\kappa}t}
\|\omega\|_{L^{\infty}(T^*\!M)}
\end{align*}
for $\omega\in L^2(T^*\!M)\cap L^{\infty}(T^*\!M)$. Therefore we obtain \eqref{eq:HessSchradarUhlenbrock3}.
\end{proof}

\begin{proof}[\bf Proof of Theorem~\ref{thm:HessShcraderUhlenbrock}]
Take $g\in {\rm Test}(M)_+$. Then, for $\omega\in D(\DD)\cap L^{\infty}(T^*\!M)$
\begin{align*}
\int_M(\Delta g-\alpha g)|\omega|\d\m-\int_M g|\omega|\d\kappa&\geq 
\int_M g\left(\left\langle \DD\omega,\frac{\omega}{|\omega|}\right\rangle -\alpha|\omega| \right)\d\m
\\
&=\int_M g \left\langle \DD\omega-\alpha\omega,\frac{\omega}{|\omega|}\right\rangle \d\m\\
&\geq-\int_M g\left|\DD\omega-\alpha\omega\right|\d\m
\end{align*}
by Lemma~\ref{lem:HKEnergyCont***}, hence
\begin{align}
\mathscr{E}_{\alpha}^{\kappa}(|\omega|,g)\leq \int_M g\left|(\alpha-\DD)\omega\right|\d\m.\label{eq:Set}
\end{align}
Since ${\rm Test}(M)_+$ is dense in $D(\mathscr{E})_+$, \eqref{eq:Set} holds for any $g\in D(\mathscr{E})_+$. 
By Lemma~\ref{lem:HessSchradarUhlenbrock2}, for $\alpha>C_{\kappa}$, we can set $\omega:=R_{\alpha}^{\rm HK}\eta\in D(\DD)\cap L^{\infty}(T^*\!M)$ for $\eta\in L^2(T^*\!M)\cap L^{\infty}(T^*\!M)$
and  $g:=R_{\alpha}^{\kappa}\psi$ with
$\psi\in L^2(M;\m)_+$. 
Then we see
\begin{align*}
(\psi,|R_{\alpha}^{\rm HK}\eta|)_{\m}\leq (\psi,R_{\alpha}^{\kappa}|\eta|)_{\m}\quad \text{ for any }\quad \psi\in L^2(M;\m)_+.
\end{align*}
This implies that for $\alpha>C_{\kappa}$ and $\eta\in  L^2(T^*\!M)\cap L^{\infty}(T^*\!M)$ 
\begin{align}
|R_{\alpha}^{\rm HK}\eta|\leq R_{\alpha}^{\kappa}|\eta|\quad \m\text{-a.e.}\label{eq:ResolventContraction}
\end{align}  
By Lemma~\ref{lem:DensenessModule} below, for any $\eta\in L^2(T^*\!M)$, there exists $\{\eta_n\}\subset 
L^2(T^*\!M)\cap L^{\infty}(T^*\!M)$ such  that $\{\eta_n\}$ converges to $\eta$ in $L^2(T^*\!M)$ and 
\eqref{eq:ResolventContraction} holds by replacing $\eta$ with $\eta_n$. 
By taking a subsequence if necessary, we can deduce that \eqref{eq:ResolventContraction} holds for general 
$\eta\in L^2(T^*\!M)$.   

From \eqref{eq:ResolventContraction}, 
we can obtain 
that $|P_t^{\rm HK}\eta|\leq P_t^{\kappa}|\eta|$ $\m$-a.e. for each $t>0$ 
 in view of the following observation: 
 \begin{align*}
 P_t^{\kappa}f&=\lim_{n\to\infty}\left(\frac{n}{t}\right)^n(R_{\frac{n}{t}}^{\kappa})^nf\quad\text{ in }L^2(M;\m)\quad\text{ for }\quad f\in L^2(M;\m),\\
  P_t^{\rm HK}\theta &=\lim_{n\to\infty}\left(\frac{n}{t}\right)^n(R_{\frac{n}{t}}^{\rm HK})^n\theta\quad
  \text{ in }L^2(T^*\!M)\quad\text{ for }\quad \theta\in L^2(T^*\!M)
   \end{align*}
  (cf.~\cite[Theorem~2.4]{ShigekawaIntertwining}). 
\end{proof}

\section{Proof of Corollary~\ref{cor:HessShcraderUhlenbrock} 
}\label{sec:cor:HessShcraderUhlenbrock}

The following lemma is stated in \cite{Braun:Tamed2021} without proof. 
We provide its proof for readers' convenience. First note that ${\rm Test}(T^*\!M)\subset L^p(T^*\!M)$ 
and ${\rm Test}(TM)\subset L^p(TM)$ for $p\in[1,+\infty]$.  

\begin{lem}\label{lem:DensenessModule}
Let $p\in[1,+\infty[$. Then ${\rm Test}(T^*\!M)$ {\rm(}resp.~${\rm Test}(TM)${\rm)} is dense in $L^p(T^*\!M)$ {\rm(}resp.~$L^p(TM)${\rm)}. Moreover, ${\rm Test}(T^*\!M)$ {\rm(}resp.~${\rm Test}(TM)${\rm)} is weakly* dense in $L^{\infty}(T^*\!M)$ {\rm(}resp.~$L^{\infty}(TM)${\rm)}.
\end{lem}
\begin{proof}[{\bf Proof}]
Assume $p\in[1,+\infty[$. In this case, we take $X\in L^p(T^*\!M)^*=L^q(TM)$ with $q:=p/(p-1)\in ]1,+\infty]$. Suppose $X\perp {\rm Test}(T^*\!M)$, i.e., 
\begin{align}
\left(\sum_{i=1}^n g_i\d f_i \right)(X)=0\quad\text{ for any }n\in\mathbb{N},\quad g_i,f_i\in {\rm Test}(M)\;\;(1\leq i\leq n).\label{eq:perp}
\end{align}
In particular, 
\begin{align*}
g\langle X,\nabla f\rangle =\langle X, g\nabla f\rangle =\left(g\d f\right)(X)=0\quad\text{ for any }\quad g,f\in{\rm Test}(M).
\end{align*}
Owing to Lemma~\ref{lem:boundedEst}, we see $(P_{1/n}h)\langle X,\nabla f\rangle =0$ for $h\in L^2(M;\m)\cap L^{\infty}(M;\m)$. 
By taking a subsequence if necessary, we may assume $h=\lim_{n\to\infty}P_{1/n}h$ $\m$-a.e., hence $h\langle X,\nabla f\rangle =0$ for such $h$. 
Let $\{M_n\}$ be an increasing sequence of Borel sets with $M=\bigcup_{n=1}^{\infty}M_n$ and 
$\m(M_n)<\infty$ for each $n\in\mathbb{N}$. 
Then for $h_n:=\1_{\{|X|\leq n\}\cap M_n}\in L^2(M;\m)\cap L^{\infty}(M;\m)$, we see $h_nX\in L^2(TM)$. 
Since $L^2(TM)$ is generated by 
$\nabla D(\mathscr{E})$ in the sense of $L^{\infty}$-modules (see \cite[Lemma~2.7]{Braun:Tamed2021}), 
we have $h_nX=0$. 
Therefore $X=0$, consequently 
${\rm Test}(T^*\!M)$ is dense in $L^p(T^*\!M)$ by \cite[Chapter III, Corollary 6.14]{Conway:FunctionalAnal}. In view of the musical isomorphisms, ${\rm Test}(TM)$ is also dense in $L^p(TM)$ for $p\in[1,+\infty[$. 

Next we assume $p=\infty$. Take $X\in L^1(TM)$ and suppose $X\perp{\rm Test}(T^*\!M)$, i.e., \eqref{eq:perp} holds. In the same way as above, we can deduce $X=0$. This means that 
${\rm Test}(T^*\!M)$ is dense in $L^{\infty}(T^*\!M)=L^1(TM)^*$ with respect to the weak* topology 
$\sigma(L^1(T\!M)^*, L^1(TM))$ by \cite[Chapter IV, Example~1.8, Corollary 3.14]{Conway:FunctionalAnal}. 
In view of the musical isomorphisms, ${\rm Test}(TM)$ is also weakly* dense in $L^{\infty}(TM)$.
\end{proof}

\begin{cor}\label{cor:LpL2dense}
Suppose $p\in[1,+\infty[$. Then $L^p(T^*\!M)\cap H^{1,2}(T^*\!M)$ 
{\rm(}resp. $L^p(TM)\cap H^{1,2}(TM)${\rm)}
is dense in $L^p(T^*\!M)$ {\rm(}resp.~$L^p(TM)${\rm)}. Moreover, $L^{\infty}(T^*\!M)\cap H^{1,2}(T^*\!M)$ 
{\rm(}resp. $L^{\infty}(TM)\cap H^{1,2}(TM)${\rm)}
is weakly* dense in $L^{\infty}(T^*\!M)$ {\rm(}resp.~$L^{\infty}(TM)${\rm)}. 
\end{cor}
\begin{proof}[{\bf Proof}]
The assertion easily follows from Lemma~\ref{lem:DensenessModule}, because ${\rm Test}(T^*\!M)\subset 
H^{1,2}(T^*\!M)$ and  ${\rm Test}(TM)\subset 
H^{1,2}(TM)$.
\end{proof}

\begin{cor}\label{cor:Reg(TM)dense}
Suppose $p\in[1,+\infty[$. Then ${\rm Reg}(T^*\!M)\cap L^p(T^*\!M)$ is dense in $L^p(T^*\!M)$ and 
${\rm Reg}(TM)\cap L^p(TM)$ is dense in $L^p(TM)$. Moreover, ${\rm Reg}(T^*\!M)\cap L^{\infty}(T^*\!M)$ is weakly* dense in $L^{\infty}(T^*\!M)$ and 
${\rm Reg}(TM)\cap L^{\infty}(TM)$ is dense in $L^{\infty}(TM)$.  
\end{cor}
\begin{proof}[{\bf Proof}]
By Lemma~\ref{lem:DensenessModule}, the assertion clearly follow from ${\rm Test}(T^*\!M)\subset {\rm Reg}(T^*\!M)\cap L^p(T^*\!M)$ and 
${\rm Test}(TM)\subset {\rm Reg}(TM)\cap L^p(TM)$. 
\end{proof}

\begin{proof}[\bf Proof of Corollary~\ref{cor:HessShcraderUhlenbrock}]
First we suppose $p\in[2,+\infty[$. 
By Theorem~\ref{thm:HessShcraderUhlenbrock} and \eqref{eq:generalpKatoEst} with $\nu=\kappa^-$, 
we easily get 
\begin{align}
\|P_t^{\rm HK}\omega\|_{L^p(T^*\!M)}\leq \|P_t^{\kappa}|\omega|\|_{L^p(M;\m)}\leq
C(\kappa)e^{C_{\kappa} t}\|\omega\|_{L^p(T^*\!M)}\label{eq:ContractionLp2}
\end{align}
holds for $\omega\in L^p(T^*\!M)\cap L^2(T^*\!M)$.
For general $\omega\in L^p(T^*\!M)$, there exists $\{\omega_n\}\subset 
 L^p(T^*\!M)\cap  H^{1,2}(T^*\!M)$ such that 
 \begin{align*}
 \lim_{n\to\infty}\|\omega_n-\omega\|_{ L^p(T^*\!M)}=0
 \end{align*}
 by Corollary~\ref{cor:LpL2dense}.   
 Then $\{P_t^{\rm HK}\omega_n\}$ forms an $L^p(T^*\!M)$-Cauchy sequence by 
 \eqref{eq:ContractionLp2}. Denote by $P_t^{\rm HK}\omega$ its limit. Since
 $P_t^{\kappa}\omega_n$ converges to $P_t^{\kappa}\omega$ in $L^p(M;\m)$, we can deduce 
 $|P_t^{\rm HK}\omega|\leq P_t^{\kappa}|\omega|$ $\m$-a.e.
  Hence $P_t^{\rm HK}$ can be extended to $L^p(T^*\!M)$  
  and 
  \eqref{eq:ContractionLp2} holds for $\omega\in L^p(T^*\!M)$. 
 In case $p=\infty$, we use another approximation. For $\omega\in L^{\infty}(T^*\!M)$, we set 
 $\omega_n:=\1_{M_n}\omega$, where $\{M_n\}$ is an increasing sequence of Borel sets with $M=\bigcup_{n=1}^{\infty}M_n$ and 
$\m(M_n)<\infty$ for each $n\in\mathbb{N}$. Then $\omega_n\in  L^2(T^*\!M)\cap L^{\infty}(T^*\!M)$ for each $n\in\mathbb{N}$, hence \eqref{eq:ContractionLp2} holds for $\omega_n$ instead of $\omega$. 
Now, for a Borel partition $(B_i)_{i\in\mathbb{N}}$ with $\m(B_i)\in]0,+\infty[$, we consider the distance ${\sf d}_{L^0(T^*\!M)}$ on $L^0(T^*\!M)$: 
\begin{align*}
{\sf d}_{L^0(T^*\!M)}(\eta_1,\eta_2):=\sum_{i\in\mathbb{N}}\frac{2^{-i}}{\m(B_i)}
\int_{B_i}\min\{|\eta_1-\eta_2|,1\}\d\m,\quad  \eta_1,\eta_2\in L^0(T^*\!M).
\end{align*}
This distance does not depend on the choice of such Borel partition $(B_i)_{i\in\mathbb{N}}$ (see \cite[1.4.3]{Braun:Tamed2021}). 
Since $P_t^{\kappa}\1_{B_i}\in L^2(M;\m)$ under $2\kappa^-\in S_{E\!K}({\bf X})$, we see from Theorem~\ref{thm:HessShcraderUhlenbrock} that 
\begin{align*}
\int_{B_i}|P_t^{\rm HK}(\omega_n-\omega_m)|\d\m&\leq\int_M P_t^{\kappa}\1_{B_i}|\omega_n-\omega_m|\d\m\\
&\leq\|\omega\|_{L^{\infty}(T^*\!M)}\int_{M_m\setminus M_n}P_t^{\kappa}\1_{B_i}\d\m\to0\quad\text{ as }\quad n,m\to\infty.
\end{align*}
This means that $(P_t^{\rm HK}\omega_n)_{n\in\mathbb{N}}$ forms a ${\sf d}_{L^0(T^*\!M)}$-Cauchy sequence. 
Denote its limit by $P_t^{\rm HK}\omega\in L^0(T^*\!M)$. Then we see $ P_t^{\rm HK}\omega\in L^{\infty}(T^*\!M)$ and \eqref{eq:ContractionLp2} holds for $\omega$.

Next we suppose $\kappa^-\in S_K({\bf X})$ and $p\in[1,+\infty[$. Then 
 \eqref{eq:generalpKatoEst} with $\nu=\kappa^-$ yields 
  \begin{align}
  \|P_t^{\kappa}\|_{p,p}\leq C(\kappa)e^{C_{\kappa}t}. 
  \label{eq:ppFeynmanKac}
  \end{align}
Thus Theorem~\ref{thm:HessShcraderUhlenbrock} and \eqref{eq:ppFeynmanKac} yield 
the estimate \eqref{eq:ContractionLp2} for $\omega\in L^p(T^*\!M)\cap L^2(T^*\!M)$, consequently 
\eqref{eq:ContractionLp2} holds for $\omega\in L^p(T^*\!M)$ as proved above. 

Finally we prove the strong continuity of $(P_t^{\rm HK})_{t\geq0}$ under $\kappa^-\in S_K({\bf X})$ and 
$p\in[1,+\infty[$. 
Suppose first $\omega\in  L^p(T^*\!M)\cap H^{1,2}(T^*\!M)$. We then know 
\begin{align}
\lim_{t\to0}\|P_t^{\rm HK}\omega-\omega\|_{H^{1,2}(T^*\!M)}=0\label{eq:StrongH12continuity}
\end{align} 
by \cite[Theorem~8.31(v)]{Braun:Tamed2021}.
In particular, 
$|P_t^{\rm HK}\omega-\omega|\to0$ in $\m$-measure as $t\to0$. 
Take any subsequence $\{t_n\}$ tending to $0$ decreasingly as $n\to\infty$ such that 
$|P_{t_n}^{\rm HK}\omega|\to|\omega|$ $n\to\infty$ $\m$-a.e. Applying Fatou's lemma, 
\begin{align}
\int_M |\omega|^p\d\m\leq\varliminf_{n\to\infty}\int_M |P_{t_n}^{\rm HK}\omega|^p\d\m.\label{eq:FormFatou}
\end{align}
On the other hand, we already know that \eqref{eq:ContractionLp2} holds for $\omega\in L^p(T^*\!M)$.
This implies 
\begin{align*}
\varlimsup_{n\to\infty}\|P_{t_n}^{\rm HK}\omega\|_{L^p(T^*\!M)}\leq C(\kappa)\|\omega\|_{L^p(T^*\!M)}.
\end{align*}
Since the coefficient $C(\kappa)$ has the following expression for some sufficiently small $t_0>0$
\begin{align*}
C(\kappa)=\frac{1}{1-\left\|\E_{\cdot}\left[A_{t_0}^{\kappa^-} \right] \right\|_{\infty}}
\end{align*}
by \eqref{eq:KatoExpression}
(see \cite[Proof of Theorem~2.2]{Sznitzman}), it can be taken to be arbitrarily close to $1$ in view of the fact that 
$\kappa^-$ is a smooth measure of Kato class. Then, we have
\begin{align*}
\varlimsup_{n\to\infty}\|P_{t_n}^{\rm HK}\omega\|_{L^p(T^*\!M)}\leq \|\omega\|_{L^p(T^*\!M)}.
\end{align*}
Combining this and \eqref{eq:FormFatou}, we have
$\lim_{n\to\infty}\|P_{t_n}^{\rm HK}\omega\|_{L^p(T^*\!M)}=\|\omega\|_{L^p(T^*\!M)}$.
Applying \cite[Theorem~16.6]{RneSchilling}, we obtain 
\begin{align*}
\lim_{n\to\infty}\|P_{t_n}^{\rm HK}\omega-\omega\|_{L^p(T^*\!M)}=0. 
\end{align*}
Thus
\begin{align*}
\lim_{t\to0}\|P_t^{\rm HK}\omega-\omega\|_{L^p(T^*\!M)}=0. 
\end{align*}
Now we consider general $\omega\in L^p(T^*\!M)$. 
By Corollary~\ref{cor:LpL2dense}, there exists $\{\omega_n\}\subset 
 L^p(T^*\!M)\cap  H^{1,2}(T^*\!M)$ such that $\lim_{n\to\infty}\|\omega_n-\omega\|_{ L^p(T^*\!M)}=0$.
By use of triangle inequality,
\begin{align*}
\|P_t^{\rm HK}\omega-\omega\|_{L^p(T^*\!M)}\leq \|P_t^{\rm HK}(\omega-\omega_n)\|_{L^p(T^*\!M)}
+\|P_t^{\rm HK}\omega_n-\omega_n\|_{L^p(T^*\!M)}+\|\omega_n-\omega\|_{ L^p(T^*\!M)}.
\end{align*}
By \eqref{eq:ContractionLp2} and the conclusion for $\omega_n\in L^p(T^*\!M)\cap L^2(T^*\!M)$, 
\begin{align*}
\varlimsup_{t\to0}\|P_t^{\rm HK}\omega-\omega\|_{L^p(T^*\!M)}\leq\left(C(\kappa)+1 \right)
\|\omega_n-\omega\|_{ L^p(T^*\!M)}\to0 \quad \text{ as }\quad n\to\infty.
\end{align*} 
Therefore
\begin{align*}
\lim_{t\to0}\|P_t^{\rm HK}\omega-\omega\|_{L^p(T^*\!M)}=0\quad \text{ for }\quad \omega\in L^p(T^*\!M).
\end{align*} 

Finally, we prove the weak* continuity of $(P_t^{\rm HK})_{t\geq0}$ on $L^{\infty}(T^*\!M)$ under $\kappa^-\in S_K({\bf X})$. 
Take $X\in L^1(TM)$. First we assume $\omega\in L^2(T^*\!M)\cap L^{\infty}(T^*\!M)$. 
Consider $X_n:=\1_{\{|X|\leq n\}\cap M_n}X$ as defined above. Then we see 
\begin{align*}
|\langle X_n,P_t^{\rm HK}\omega-\omega\rangle |\leq\|X_n\|_{L^1(TM)}\|P_t^{\rm HK}\omega-\omega\|_{L^2(T^*\!M)}\to0\quad\text{ as }\quad t\to0.
\end{align*}
From this, 
\begin{align*}
\varlimsup_{t\to0}|\langle X,P_t^{\rm HK}\omega-\omega\rangle |&\leq
\varlimsup_{t\to0}\|X-X_n\|_{L^1(TM)}\|P_t^{\rm HK}\omega-\omega\|_{L^{\infty}(T^*\!M)}+\varlimsup_{t\to0}|\langle X_n,P_t^{\rm HK}\omega-\omega\rangle |\\
&\leq \|X-X_n\|_{L^1(TM)}(1+C(\kappa)))\|\omega\|_{L^{\infty}(T^*\!M)}\to0\quad\text{ as }\quad n\to\infty.
\end{align*}
Then, we have $\lim_{t\to0}\langle X_n,P_t^{\rm HK}\omega-\omega\rangle =0$ for $\omega\in L^2(T^*\!M)\cap L^{\infty}(T^*\!M)$. Next we take general $\omega\in L^{\infty}(T^*\!M)$. Consider $\omega_n:=\1_{M_n}\omega$ as above. Then we know $P_t^{\rm HK}\omega\in L^2(T^*\!M)\cap L^{\infty}(T^*\!M)$. Then we have
\begin{align*}
|\langle X,\omega_n-\omega\rangle |&\leq\int_M|\langle X,\1_{M_n^c}\omega\rangle_{\m} |\d\m
\leq\|\omega\|_{L^{\infty}(T^*\!M)}\int_{M_n^c}|X|_{\m}\d\m\to0\quad \text{ as }\quad n\to\infty,\\
|\langle X,P_t^{\rm HK}(\omega_n-\omega)\rangle |&\leq\int_M|\langle X,P_t^{\rm HK}\1_{M_n^c}\omega\rangle_{\m} |\d\m
\\&\leq\|\omega\|_{L^{\infty}(T^*\!M)}\int_{M_n^c}P_t^{\kappa}|X|_{\m}\d\m\to0\quad \text{ as }\quad n\to\infty,
\end{align*}
where we use the $\m$-symmetry of $P_t^{\kappa}$ and $P_t^{\kappa}|X|\in L^1(M;\m)$ under $\kappa^-\in S_K({\bf X})$. Combining this inequality and 
\begin{align*}
|\langle X,P_t^{\rm HK}\omega-\omega\rangle |&\leq 
|\langle X,P_t^{\rm HK}(\omega-\omega_n)\rangle |+ |\langle X,P_t^{\rm HK}\omega_n-\omega_n)\rangle |+ |\langle X,\omega_n-\omega)\rangle |, 
\end{align*}
we obtain 
\begin{align*}
\varlimsup_{t\to0}|\langle X,P_t^{\rm HK}\omega-\omega\rangle |\leq\|\omega\|_{L^{\infty}(TM)}\int_{M_n^c}(|X|_{\m}+P_t^{\kappa}|X|_{\m})\d\m\to0\quad\text{ as }\quad n\to\infty.
\end{align*}
Therefore, $\lim_{t\to0}\langle X,P_t^{\rm HK}\omega-\omega\rangle =0$ for $X\in L^1(TM)$ and $\omega\in L^{\infty}(T^*\!M)$. This implies the weak* continuity of $(P_t^{\rm HK})_{t\geq0}$ on $L^{\infty}(T^*\!M)=L^1(TM)^*$. 
\end{proof}

\section{Examples}\label{sec:examples}
\begin{example}[Riemannian manifolds with boundary]
{\rm \quad Let $(M,g)$ be a smooth Riemannian manifold with boundary $\partial M$. 
Denote by $\mathfrak{v}:={\rm vol}_g$ the Riemannian volume measure induced by $g$, and by 
$\mathfrak{s}$ the surface measure on $\partial M$
(see \cite[\S1.2]{Braun:Tamed2021}). 
If $\partial M\ne \emptyset$, then $\partial M$ is a smooth co-dimension $1$ submanifold of $M$ and it becomes Riemannian 
when endowed with the pulback metric 
\begin{align*}
\langle \cdot,\cdot\rangle_j:=j^*\langle \cdot,\cdot\rangle,\qquad \langle u,v:=g(u,v)\quad\text{ for }\quad u,v\in TM.
\end{align*}  
under the natural inclusion $j:\partial M\to M$. The map $j$ induces a natural inclusion $d_j:T\partial M\to TM|_{\partial M}$ which is not surjective. In particular, the vector bundles $T\partial M$ and $TM|_{\partial M}$ do not coincide. 

Let $\m$ be a Borel measure on $M$ which is locally equivalent to $\mathfrak{v}$. 
Let $D(\mathscr{E}):=W^{1,2}(M^{\circ})$ be the Sobolev space with respect to 
$\m$ defined in the usual sense on $M^{\circ}:=M\setminus\partial M$. Define $\mathscr{E}:W^{1,2}(M^{\circ})\to [0,+\infty[$ 
by 
\begin{align*}
\mathscr{E}(f):=\int_{M^{\circ}}|\nabla f|^2\d \m
\end{align*}  
and the quantity $\mathscr{E}(f,g)$, $f,g\in W^{1,2}(M^{\circ})$, by polarization. Then $(\mathscr{E}, D(\mathscr{E}))$ becomes a strongly local regular Dirichlet form on $L^2(M;\m)$, since $C_c^{\infty}(M)$ is a dense set of $D(\mathscr{E})$. 
Let $k:M^{\circ}\to\R$ and $\ell:\partial M\to\R$ be continuous functions providing lower bounds on the Ricci curvature and the second fundamental form of $\partial M$, respectively. 
Suppose that $M$ is compact and $\m=\mathfrak{v}$. Then $(M,\mathscr{E},\m)$ is tamed by 
\begin{align*}
\kappa:=k\mathfrak{v}+\ell\mathfrak{s},
\end{align*}
because $\mathfrak{v},\mathfrak{s}\in S_K({\bf X})$ (see \cite[Theorem~2.36]{ERST} and \cite[Theorem~5.1]{Hsu:2001}). Then one can apply Theorem~\ref{thm:HessShcraderUhlenbrock} and Corollary~\ref{cor:HessShcraderUhlenbrock} to $(M,\mathscr{E},\m)$. 

More generally, if $M$ is regularly exhaustible, i.e., there exists an increasing sequence $(M_n)_{n\in\N}$ of domains $M_n\subset M^{\circ}$ with smooth boundary $\partial M_n$ such that $g$ is smooth on $M_n$ and the following properties hold:
\begin{enumerate}
\item[(1)] The closed sets $(\overline{M}_n)_{n\in\N}$ constitute an $\mathscr{E}$-nest for $(\mathscr{E},W^{1,2}(M))$.  
\item[(2)] For all compact sets $K\subset M^{\circ}$ there exists $N\in\N$ such tht $K\subset M_n$ for all $n\geq N$.
\item[(3)] There are lower bounds $\ell_n:\partial M_n\to\R$ for the curvature of $\partial M_n$ with $\ell_n=\ell$ on $\partial M\cap \partial M_n$ such that the distributions $\kappa_n=k\mathfrak{v}_n+\ell_n\mathfrak{s}_n$ are uniformly $2$-moderate in 
the sense that 
\begin{align}
\sup_{n\in\N}\sup_{t\in[0,1]}\sup_{x\in M_n}\E^{(n)}\left[e^{-2A_t^{\kappa_n}} \right]<\infty,\label{eq:2moderate}
\end{align}
where $\mathfrak{v}_n$ is the volume measure of $M_n$ and $\mathfrak{s}_n$ is the surface measure of $\partial M_n$.
\end{enumerate}
Suppose $\m=\mathfrak{v}$. 
Then $(M,\mathscr{E},\m)$ is tamed by $\kappa=k\mathfrak{v}+\ell\mathfrak{s}$ 
(see \cite[Theorem~4.5]{ERST}). In this case, 
Theorem~\ref{thm:HessShcraderUhlenbrock} and Corollary~\ref{cor:HessShcraderUhlenbrock} hold for $(M,\mathscr{E},\m)$ provided 
$\kappa^+\in S_D({\bf X})$, $\kappa^-\in S_{K}({\bf X})$. 

Let $Y$ be the domain defined by 
\begin{align*}
Y:=\{(x,y,z)\in\R^3\mid z>\phi(\sqrt{x^2+y^2})\},
\end{align*}
where $\phi:[0,+\infty[\to[0,+\infty[$ is $C^2$ on $]0,+\infty[$ with $\phi(r):=r-r^{2-\alpha}$, $\alpha\in]0,1[$ for 
$r\in[0,1]$, $\phi$ is a constant for $r\geq2$ and $\phi''(r)\leq0$ for $r\in[0,+\infty[$. 
Let $\m_Y$ be the $3$-dimensional Lebesgue measure restricted to $Y$ and 
$\sigma_{\partial Y}$ the $2$-dimensional Hausdorff measure on $\partial Y$. Denote by 
$\mathscr{E}_Y$ the Dirichlet form on $L^2(Y;\m)$ with Neumann boundary conditions. 
The smallest eigenvalue of the second fundamental form of $\partial Y$ can be given by 
\begin{align*}
\ell(r,\phi)=\frac{\phi''(r)}{(1+|\phi'(r)|^2)^{3/2}} (\leq0),
\end{align*}
for $r\leq 1$ and $\ell=0$ for $r\geq2$. It is proved in \cite[Theorem~4.6]{ERST} that 
the Dirichlet space $(Y,\mathscr{E}_Y,\m_Y)$ is tamed by 
\begin{align*}
\kappa=\ell\sigma_{\partial Y}.
\end{align*}
In this case, the measure-valued Ricci lower bound $\kappa$ satisfies 
$|\kappa|=|\ell|\sigma_{\partial Y}\in S_K({\bf X})$ (see \cite[Lemma~2.34, Theorem~2.36, Proof of Theorem~4.6]{ERST}). 
Then 
Theorem~\ref{thm:HessShcraderUhlenbrock} and Corollay~\ref{cor:HessShcraderUhlenbrock} hold for $(Y,\mathscr{E}_Y,\m_Y)$. 
}
\end{example}
\begin{example}[Configuration space over metric measure spaces]
{\rm Let $(M,g)$ be a complete smooth Riemannian manifold without boundary.   
The configuration space $\Upsilon$ over $M$ is the space of all locally finite point measures, that is, 
\begin{align*}
\Upsilon:=\{\gamma\in\mathcal{M}(M)\mid \gamma(K)\in\N\cup\{0\}\quad \text{ for all compact sets}\quad K\subset M\}. 
\end{align*}
In the seminal paper Albeverio-Kondrachev-R\"ockner~\cite{AKR} identified a natural geometry on $\Upsilon$ by lifting the geometry of $M$ to $\Upsilon$. In particular, there exists a natural gradient $\nabla^{\Upsilon}$, divergence ${\rm div}^{\Upsilon}$ and 
Laplace operator $\Delta^{\Upsilon}$ on $\Upsilon$. It is shown in \cite{AKR} that the Poisson point measure $\pi$ on 
$\Upsilon$ is the unique (up to intensity) measure on $\Upsilon$ under which the gradient and divergence become 
dual operator in $L^2(\Upsilon;\pi)$. Hence, the Poisson measure $\pi$ is the natural volume measure on $\Upsilon$ 
and $\Upsilon$ can be seen as an infinite dimensional Riemannian manifold. The canonical Dirichlet form 
\begin{align*}
\mathscr{E}(F)=\int_{\Upsilon}|\nabla^{\Upsilon}F|_{\gamma}^2\pi(\d\gamma)
\end{align*}
 constructed in \cite[Theorem~6.1]{AKR} is quasi-regular and strongly local
 and it induces the heat semigroup $T_t^{\Upsilon}$ and a Brownian motion ${\bf X}^{\Upsilon}$ on $\Upsilon$ which can be identified with the independent infinite particle process. If ${\rm Ric}_g\geq K$ on $M$ with $K\in \R$, then $(\Upsilon,\mathscr{E}^{\Upsilon},\pi)$ is tamed by 
 $\kappa:=K\pi$ with $|\kappa|\in S_K({\bf X}^{\Upsilon})$ (see \cite[Theorem~4.7]{EKS} and \cite[Theorem~3.6]{ERST}). 
 Then one can apply Theorem~\ref{thm:HessShcraderUhlenbrock} and Corollary~\ref{cor:HessShcraderUhlenbrock} for $(\Upsilon,\mathscr{E}^{\Upsilon},\pi)$. 
 
 More generally, in Dello Schiavo-Suzuki~\cite{DelloSuzuki:ConfigurationI}, configuration space $\Upsilon$ over proper complete and separable metric space $(M,{\sf d})$ is considered. The configuration space $\Upsilon$ is endowed with the \emph{vage topology} $\tau_V$, induced by duality with continuous compactly supported functions on $M$, and with a reference Borel probability measure $\mu$ satisfying \cite[Assumption~2.17]{DelloSuzuki:ConfigurationI}, commonly understood as the law of a proper point process on $M$. In \cite{DelloSuzuki:ConfigurationI}, 
 they constructed the strongly local Dirichlet form $\mathscr{E}^{\Upsilon}$ defined to be the $L^2(\Upsilon;\mu)$-closure of a certain pre-Dirichlet form on a class of certain cylinder functions and prove its quasi-regularity for a wide class of measures $\mu$ and base spaces (see \cite[Proposition~3.9 and Theorem~3.45]{DelloSuzuki:ConfigurationI}). Moreover, in 
 Dello Schiavo-Suzuki~\cite{DelloSuzuki:ConfigurationII}, for any fixed $K\in \R$ they prove that a Dirichlet form $(\mathscr{E},D(\mathscr{E}))$ with its 
 carr\'e-du-champ $\Gamma$ satisfies ${\sf BE}_2(K,\infty)$ if and only if the Dirichlet form $(\mathscr{E}^{\Upsilon},D(\mathscr{E}^{\Upsilon}))$ on $L^2(\Upsilon;\mu)$ with its carr\'e-du-champ $\Gamma^{\Upsilon}$ satisfies ${\sf BE}_2(K,\infty)$.
Hence, if $(M,\mathscr{E},\m)$ is tamed by $K\m$ with $|K|\m\in S_K({\bf X})$, then  $(\Upsilon,\mathscr{E}^{\Upsilon},\mu)$ is tamed by $\kappa:=K\mu$ with $|\kappa|\in S_K({\bf X}^{\Upsilon})$. 
Then Theorem~\ref{thm:HessShcraderUhlenbrock} and Corollary~\ref{cor:HessShcraderUhlenbrock} hold for $(\Upsilon,\mathscr{E}^{\Upsilon},\mu)$ 
under the suitable class of measures $\mu$ defined in \cite[Assumption~2.17]{DelloSuzuki:ConfigurationI}.  
}
\end{example}

\noindent
{\bf Acknowledgment.} 
The author would like to thank Professor Mathias Braun for notifying us  
the reference \cite{Braun:HeatFlow} after the first draft of this paper. His comment and advice helped us  so much to improve the quality of this paper. 
\bigskip

\noindent
{\bf Conflict of interest.} The authors have no conflicts of interest to declare that are relevant to the content of this article.

\bigskip
\noindent
{\bf Data Availability Statement.} Data sharing is not applicable to this article as no datasets were generated or analyzed during the current study.

\providecommand{\bysame}{\leavevmode\hbox to3em{\hrulefill}\thinspace}
\providecommand{\MR}{\relax\ifhmode\unskip\space\fi MR }
\providecommand{\MRhref}[2]{%
  \href{http://www.ams.org/mathscinet-getitem?mr=#1}{#2}
}
\providecommand{\href}[2]{#2}

\end{document}